\documentclass[a4paper, 12pt]{article}

\usepackage{graphicx}
\usepackage{amsmath,amsthm,amssymb,mathrsfs}

\theoremstyle{plain}
\newtheorem{theorem}{Theorem}[section]
\newtheorem{lemma}[theorem]{Lemma}
\newtheorem{proposition}[theorem]{Proposition}

\theoremstyle{definition}
\newtheorem{definition}[theorem]{Definition}
\newtheorem{remark}[theorem]{Remark}
\newtheorem{example}[theorem]{Example}

\numberwithin{equation}{section}

\allowdisplaybreaks[3]

\title{Ill-Posedness of the Third Order NLS Equation\\
with Raman Scattering Term}
\author{Nobu Kishimoto$\dagger$ and Yoshio Tsutsumi$\ddagger$}
\date{$\dagger$ RIMS, Kyoto University, Kyoto 606-8502, JAPAN \\
$\ddagger$ Department of Mathematics, Kyoto University, \\ Kyoto 606-8502, JAPAN}

\begin{document}

\maketitle

\begin{abstract}
We consider the ill-posedness and well-posedness of the Cauchy problem for the third order NLS equation with Raman scattering term on the one dimensional torus. It is regarded as a mathematical model for the photonic crystal fiber oscillator. Regarding the ill-posedness, we show the nonexistence of solutions in the Sobolev space and the norm inflation of the data-solution map under slightly different conditions, respectively. We also prove the local unique existence of solutions in the analytic function space.
\end{abstract}


%
%

\section{Introduction and Main Theorems}

In the present paper, we consider the ill-posedness of the Cauchy problem for the nonlinear Schr\"odinger equation with third order dispersion and intrapulse Raman scattering term (see (2.3.43) on page 40 of \cite{Agr}):
\begin{align}
     &\partial_t u  = \alpha_1 \partial_x^3 u + i \alpha_2 \partial_x^2 u + i\gamma_1 |u|^2 u +\gamma_2 \partial_x \bigl ( |u|^2 u \bigr ) - i\Gamma u \partial_x \bigl (|u|^2\bigr ),  \label{3NLSR} \\
    &\hskip 5.5cm t \in [-T, T], \quad x \in \mathbf{T} = \mathbf{R}/2 \pi \mathbf{Z},  \nonumber \\
   &u(0, x) = u_0(x), \qquad x \in \mathbf{T}, \label{ic}
\end{align}
where $\alpha_j$, $\gamma_j$ $(j = 1, 2)$ and $\Gamma$  are real constants and $T$ is a positive constant.
Throughout this paper, we assume that
\begin{equation}
\Gamma >0,\qquad \alpha_1 \neq 0,\qquad \frac{2\alpha_2}{3\alpha_1}\not\in \mathbf{Z}.  \label{h1}
\end{equation}
The last and the last but one terms on the right-hand side of \eqref{3NLSR} represent the effect of the intrapulse Raman scattering, which is not negligible for ultrashort optical pulses (see \cite [\S 2.3.2] {Agr}).
The well-posedness in the Sobolev space $H^s$ of the Cauchy problem \eqref{3NLSR} and \eqref{ic} without Raman scattering terms has been intensively studied by Miyaji and the second author \cite{MT1, MT2}.
It is showed that the Cauchy problems of \eqref{3NLSR} and the reduced equation relevant to \eqref{3NLSR} are well-posed in $H^s$, $s \geq 0$ and $s > -1/6$, respectively, in \cite{MT1} and \cite{MT2} (for the reduced equation, see \eqref{red3NLS} below).
For the Cauchy problem \eqref{3NLSR} and \eqref{ic} with the coefficient of the last term $u\partial _x(|u|^2)$ being real rather than imaginary, Takaoka~\cite{Tak} showed the well-posedness in Sobolev spaces $H^s$ for $s\ge 1/2$.
In the present paper, we show that the last term on the right-hand side of \eqref{3NLSR} causes the ill-posedness of the Cauchy problem \eqref{3NLSR} and \eqref{ic}.

Now we briefly explain how the last term on the right-hand side of \eqref{3NLSR} causes the ill-posedness.
We divide the Raman scattering term into the nonresonant and the resonant parts:
\begin{align*}
   \widehat{\bigl [ i u \partial_x |u|^2 \bigr ]} (k)
   &= - \frac{1}{2\pi} \sum_{\begin{subarray}{c} k = k_1 + k_2 + k_3 \\ k_1 + k_2 \neq 0 \end{subarray}} (k_1+k_2) \hat u(k_1) \hat {\bar u}(k_2) \hat u(k_3) \\
   &= - \frac{1}{2\pi} \sum_{(k_1+k_2)(k_2+k_3) \neq 0} - \frac{1}{2\pi} \sum_{\begin{subarray}{c} k_1 + k_2 \neq 0 \\ k_2 + k_3 = 0 \end{subarray}} =: \hat I_1 + \hat I_2,
\end{align*}
where $\hat f$ denotes the Fourier coefficient of $f$ in the $x$ variable (see \eqref{defn:fc} below for the precise definition).
Here, we note that $\hat I_1$ is the nonresonant part and $\hat I_2$ is the resonant part.
We can rewrite $\hat I_2$ as follows.
\begin{align*}
   2\pi \hat I_2 = &- k \hat u(k) \sum_{\begin{subarray}{c} k_2 \neq -k \\ k_2 + k_3 = 0 \end{subarray}} \hat{\bar u}(k_2) \hat u(k_3) 
- \hat u(k) \sum_{\begin{subarray}{c} k_2 \neq -k \\ k_2 + k_3 = 0 \end{subarray}} k_2 \hat{\bar u}(k_2) \hat u(k_3) \\
   = & - k \hat u(k) \sum_{k_2 \in \mathbf{Z}} \hat{\bar u}(k_2) \hat u(-k_2) + k \hat u(k) \hat{\bar u}(-k) \hat u(k) \\
   &- \hat u(k) \sum_{k_2 \in \mathbf{Z}} k_2 \hat{\bar u}(k_2) \hat u(-k_2) - k \hat u(k) \hat{\bar u}(-k) \hat u (k) \\
   = & - k \hat u(k) \sum_{k_2 \in \mathbf{Z}} |\hat u(k_2)|^2 + \hat u(k) \sum_{k_2 \in \mathbf{Z}} k_2 |\hat u(k_2)|^2 \\
   =&  -k \| u \|_{L^2}^2 \hat u(k) + \hat u(k) \sum_{k_2 \in \mathbf{Z}} k_2 |\hat u(k_2)|^2.
\end{align*}
At the last equality but one, we have used the fact that  $\hat u (-k) = \bar{\hat{\bar u}}(k)$.
Therefore, we obtain
\[
   I_2 = \frac{i}{2\pi} \| u \|_{L^2}^2 \partial_x u + \frac{1}{2\pi} \Bigl ( \sum_{k_2 \in \mathbf{Z}} k_2 |\hat u(k_2)|^2 \Bigr ) u.
\]
Hence, since the $L^2$ norm is conserved (see Lemma \ref{lem:L2law} in \S 2), the equation \eqref{3NLSR} can be rewritten as follows:
\begin{align}
     \partial_t u +& i a  \partial_x u = \alpha_1 \partial_x^3 u + i \alpha_2 \partial_x^2 u + i \gamma_1 |u|^2 u + \ i \gamma_2 \partial_x \bigl ( |u|^2 u \bigr )  \label{3NLSR2} \\
     &+ \frac{\Gamma}{(2\pi )^{3/2}} \sum_{k \in \mathbf{Z}} e^{-ikx} \sum_{(k_1+k_2)(k_2+k_3) \neq 0} (k_1 + k_2) \hat u(k_1) \hat {\bar u}(k_2) \hat u(k_3)  \notag \\
     &- \frac{\Gamma}{2\pi} \bigl ( \sum_{k_2 \in \mathbf{Z}} k_2 |\hat u(k_2)|^2 \bigr ) u,  \qquad t \in [-T, T], \  x \in \mathbf{T}, \notag 
\end{align}
where
\[
     a = \frac{\Gamma}{2\pi}  \|u_0\|_{L^2}^2.
\]
Consequently, the Cauchy-Riemann type elliptic operator $\partial_t + i a \partial_x$ appears due to the Raman scattering term.
On the other hand, $I_1$ can be estimated in $H^s$ for $s \geq 1/2$ (see Bourgain \cite[Section 8.I]{Bour1}).
This observation suggests that the Cauchy problem \eqref{3NLSR} and \eqref{ic} should be ill-posed.

Before stating the main theorems in this paper, we define the solution of \eqref{3NLSR} and \eqref{ic}.

\begin{definition} \label{def}
Let $T>0$, $s\ge 0$, and $u_0\in H^s(\mathbf{T})$.
We say $u$ is a solution to the Cauchy problem \eqref{3NLSR} and \eqref{ic} on $[0,T)$ if $u$ satisfies
\[ u\in L^\infty _{loc}([0,T);H^s(\mathbf{T}))\cap L^3_{loc}([0,T)\times \mathbf{T}),\quad \partial _x(|u|^2)\in L^1_{loc}([0,T);H^{-s}(\mathbf{T})),\]
and if \eqref{3NLSR} and \eqref{ic} hold in the sense of distribution; \mbox{i.e.},
\begin{align*}
&-\int _0^T\int _{\mathbf{T}}u\partial _t\phi \,dx\,dt -\int _{\mathbf{T}}u_0\phi (0,\cdot )\,dx \\
&\quad =\int _0^T\int _{\mathbf{T}}\big [ u\bigl ( -\alpha_1 \partial _x^3\phi +i\alpha_2 \partial _x^2\phi \bigr )+|u|^2u\bigl ( i\gamma_1 \phi -\gamma_2 \partial _x\phi \bigr ) \bigr ] \,dx\,dt\\
&\qquad -i\Gamma \int _0^T\bigl \langle \partial _x\bigl (|u|^2\bigr ) (t)\,,\,u(t)\phi (t)\bigr \rangle _{H^{-s},H^s}\,dt
\end{align*}
for any $\phi \in C^\infty _0([0,T)\times \mathbf{T})$.
We say $u$ is a solution on a closed interval $[0,T]$ if $u\in C([0,T];H^s(\mathbf{T}))$ for some $s>1/2$ and $u|_{[0,T)}$ is a solution on $[0,T)$ in the above sense.
A solution on $(-T,0]$, $[-T,0]$ is defined in a similar manner.
\end{definition}

\begin{remark}\label{remsol}
(a) If $u(t,x)$ is a solution to \eqref{3NLSR} and \eqref{ic} on $[0,T)$, then $\overline{u(-t,x)}$ is a solution on $(-T,0]$ to \eqref{3NLSR} and \eqref{ic} with $(\alpha_1,\gamma_2)$ replaced by $(-\alpha_1,-\gamma_2)$.

(b) Let $u$ be a solution on $[0,T)$ to the Cauchy problem \eqref{3NLSR} and \eqref{ic} with $u_0\in H^s(\mathbf{T})$. 
If $s>1/2$ and $u\in C([0,T);H^s(\mathbf{T}))$, then one can show by the Sobolev inequalities that $u\in C^1([0,T);H^{s-3}(\mathbf{T}))$ and \eqref{3NLSR} is satisfied in $C([0,T);H^{s-3}(\mathbf{T}))$, while \eqref{ic} is verified in $H^s(\mathbf{T})$.
In particular, $\hat{u}(\cdot ,k)\in C^1([0,T))$ for any $k\in \mathbf{Z}$ and we have
\begin{align}
\begin{split}
\partial _t\hat{u}(t,k)
=& -i(\alpha_1 k^3+\alpha_2 k^2)\hat{u}(t,k) \\
&+ \sum _{k=k_1+k_2+k_3}\frac{i\gamma_1 +i\gamma_2 k+\Gamma (k_1+k_2)}{2\pi}\hat{u}(t,k_1)\hat{\bar{u}}(t,k_2)\hat{u}(t,k_3),
\end{split}\label{hateq} \\
\hat{u}(0,k)=&\; \hat{u}_0(k)\label{hatic}
\end{align}
for any $k\in \mathbf{Z}$, where the summation on the right-hand side of \eqref{hateq} converges absolutely and uniformly in $t$ on any compact subinterval of $[0,T)$.

(c) By (b) and the continuity at $t=T$, a solution $u$ on a closed interval $[0,T]$ satisfies $\hat{u}(\cdot ,k)\in C^1([0,T])$ for any $k\in \mathbf{Z}$ and \eqref{hateq}, \eqref{hatic} for any $(t,k)\in [0,T]\times \mathbf{Z}$.
\end{remark}

We have the following two theorems concerning the ill-posedness of the Cauchy problem \eqref{3NLSR} and \eqref{ic}.

\begin{theorem} \label{ip0} 
We assume that \eqref{h1} holds.
For any $s\ge 1$, there exists $u_0\in H^s (\mathbf{T})$ such that for \emph{no} $T>0$ the Cauchy problem \eqref{3NLSR} and \eqref{ic} has a solution $u\in C([0,T];H^s(\mathbf{T}))$ on $[0,T]$, or a solution $u\in C([-T,0];H^s(\mathbf{T}))$ on $[-T,0]$.
\end{theorem}

\begin{remark}
In fact, we can show the nonexistence of solutions in a larger class; i.e., in $C_tH^{s_1}_x$ for some $s_1<s$.
Moreover, a slightly weaker nonexistence result holds even for some $C^\infty$ initial data.
See Theorem~\ref{ip} and Theorem~\ref{ip+} below for the precise statements.
\end{remark}

\begin{remark}
In the case of $\mathbf{R}^n$, the Cauchy problem of the semilinear Schr\"odinger equation is well-posed in regular Sobolev spaces (see Hayashi and Ozawa \cite{HO} and Chihara \cite{Ch1} for the one dimensional case and see Chihara \cite{Ch2} for the higher dimensional case).
The same is true of the third order NLS with Raman scattering term (see Staffilani \cite{St}).
It is in sharp contrast to our case of $\mathbf{T}$.
The difference between the cases of $\mathbf{R}$ and $\mathbf{T}$ is that the spectrum of the Laplacian is continuous in the former case, while it is discrete in the latter case.
\end{remark}

\begin{remark}\label{rem:NLS}
The same nonexistence result as Theorem \ref{ip0} holds for the equation \eqref{3NLSR} with $\alpha_1=0$ (see Proposition \ref{prop:NLS} below).
\end{remark}

\begin{theorem}\label{infl0} 
We assume that \eqref{h1} holds.
Then, for any $s\ge 1$, inflation of the $H^s$ norm occurs around any $H^s$ solution in the following sense:
Let $u^*\in C([0,T];H^s(\mathbf{T}))$ be a solution to \eqref{3NLSR} on $[0,T]$ for some $T>0$.
Then, for any $\varepsilon >0$ and $0<\tau \le T$ there exists a real analytic function $\phi$ on $\mathbf{T}$ with $\| \phi \|_{H^s}\le \varepsilon$ such that either there does not exist a solution $u$ to \eqref{3NLSR} on $[0,\tau ]$ with the initial condition $u(0)=u^*(0)+\phi$ in the class $C([0,\tau ];H^s(\mathbf{T}))$, or such a solution exists but 
\[ \sup _{t\in [0,\tau ]}\| u(t)-u^*(t)\| _{H^s}\ge \varepsilon ^{-1}.\]
\end{theorem} 

\begin{remark}
From Theorem \ref{ip0}, it seems impossible to solve the Cauchy problem \eqref{3NLSR} and \eqref{ic} in Sobolev spaces.
But there is still a chance that the Cauchy problem \eqref{3NLSR} and \eqref{ic} is solvable for a class of $C^\infty$ initial data.
Even if it is the case, Theorem \ref{infl0} shows that the solution map$:u_0 \mapsto u$ is discontinuous everywhere, which implies the continuous dependence of solutions on initial data breaks down in Sobolev spaces.
\end{remark}

\begin{remark}
Similarly to the nonexistence result, we can in fact show inflation of the $H^{s_1}$ norm for some $s_1<s$.
See Theorem~\ref{infl} and Theorem~\ref{infl+} below for the precise statements.
\end{remark}

\begin{remark} \label{af}
A large number of numerical simulations for the Cauchy problem \eqref{3NLSR}--\eqref{ic} have been made though it is ill-posed in Sobolev spaces (see, e.g., \cite{Agr}).
In those numerical computations, such analytic functions as Gaussian and super-Gaussian pulses are chosen as initial data.
So, it is natural to expect that the Cauchy problem \eqref{3NLSR}--\eqref{ic} should be solvable in the analytic function space.
Indeed, we describe the result on the unique solvability in the analytic function space in Section 4 (see Proposition \ref{prop:analytic} below).
\end{remark}

There are many papers concerning the well-posedness issue for the Cauchy problem of nonlinear dispersive equations (see, e.g., \cite{Bour1}, \cite{Chr2}, \cite{CCT}, \cite{ErdTz}, \cite{GKO}, \cite{GO}, \cite{KPV1, KPV2}, \cite{KTzv}, \cite{KwOh}, \cite{MT1, MT2}, \cite{Mol}, \cite{MPV}, \cite{NTT}, \cite{St}, \cite{Tak}, \cite{TT} and \cite{Tzv}).
For the well-posedness of linear Schr\"odinger equations, Mizohata \cite{Miz} and Chihara \cite{Ch} studied necessary and sufficient conditions in the cases of $\mathbf{R}^n$ and $\mathbf{T}^n$, respectively.
In \cite{Ch}, Chihara also treated the ill-posedness of the nonlinear Schr\" odinger equation.
These works on linear equations give deep insight to nonlinear dispersive equations.
On the other hand, in the nonlinear case, a linearized equation can not determine all properties of the original nonlinear equation.
Indeed, the Cauchy-Riemann type operator on the left-hand side of \eqref{3NLSR2} does not immediately imply the ill-posedness of \eqref{3NLSR} and \eqref{ic}.
This is because the singularity caused by the nonlinearity might cancel out the one appearing in the Cauchy-Riemann type operator.
Therefore, we need to estimate the balance between the singularities to which the nonlinearity and the Cauchy-Riemann type operator give rise.
For that purpose, we use the smoothing type effect for the cubic nonlinearity of such nonlinear dispersive equations as the mKdV, the NLS and the third order NLS equations on the one dimensional torus (see, e.g., \cite{ErdTz}, \cite{GKO}, \cite{GO} and \cite{Mol} for the NLS equation, \cite{KwOh}, \cite{MPV}, \cite{NTT} and \cite{TT} for the mKdV equation, and \cite{MT1, MT2} for the third order NLS).
The estimate based on the smoothing type effect enables us to show that the singularity coming from the Cauchy-Riemann type operator is dominant over the one caused by the nonlinearity.

We should here make a remark on whether or not we can recover the well-posedness for the reduced equation relevant to \eqref{3NLSR}.
This is because for the well-posedness issue, we often consider the reduced equation, which is derived from the elimination of bad terms from the original equation (see, e.g., \cite{Bour1}, \cite{KPV1}, \cite{KwOh} and \cite{MT2}).
If we put 
\[
   v(t, x) = u\bigl ( t, x - \frac {\gamma_2} {\pi} \int_0^t \|u(s)\|_{L^2}^2 \, ds \bigr ) e^{-\frac {\gamma_1} \pi i \int_0^t \|u(s)\|_{L^2}^2 \, ds},
\]
then the reduced equation can be formally written as follows.
\begin{align}
   \partial_t v  + i a \partial_x v =&\; \alpha_1 \partial_x^3 v + i \alpha_2 \partial_x^2 v +i \gamma_1 \bigl ( |v|^2 - \frac 1 \pi \| v(t) \|_{L^2}^2  \bigr ) v \label{red3NLS} \\
   &+ \gamma_2 \bigl [ 2 \bigl ( |v|^2 - \frac 1 {2 \pi} \| v(t) \|_{L^2}^2  \bigr ) \partial_x v + v^2 \partial_x \bar v \bigr ]    \nonumber \\
   &+ \frac{\Gamma}{(2\pi )^{3/2}} \sum_{k \in \mathbf{Z}} e^{-ikx} \hspace{-10pt}\sum_{(k_1+k_2)(k_2+k_3) \neq 0} \hspace{-10pt}(k_1 + k_2) \hat v(k_1) \hat {\bar v}(k_2) \hat v(k_3) \notag \\
   &- \frac{\Gamma}{2\pi} \bigl ( \sum_{k_2 \in \mathbf{Z}} k_2 |\hat v(k_2)|^2 \bigr ) v, \qquad t \in [-T, T], \quad x \in \mathbf{T}.  \nonumber
\end{align}
On the left-hand side of \eqref{red3NLS}, the Cauchy-Riemann type operator appears and so Theorem \ref{ip0} also holds for \eqref{red3NLS}.
In fact, we use \eqref{red3NLS} to prove the ill-posedness (see \eqref{eq:hat-w} and Remark \ref{redeq} below).

The plan of this paper is as follows.
In Section 2,  we give several lemmas needed for the proofs of Theorems \ref{ip0} and \ref{infl0} and prove Theorem \ref{ip0}.
We also show that Theorem \ref{ip0} holds for the Schr\" odinger equation with derivative nonlinearity, that is, for the case of $\alpha_1 = 0$ (see Proposition \ref{prop:NLS} below).
In Section 3, we give the proof of Theorem \ref{infl0}.
Finally, in Section 4, we describe the unique local solvability of the Cauchy problem \eqref{3NLSR} and \eqref{ic} in the analytic function space (see Proposition \ref{prop:analytic} below).

We conclude this section with notation given.
We use the following definition of the Fourier coefficients of functions on $\mathbf{T}$: 
\begin{equation}\label{defn:fc}
\hat{f}(k):= \frac 1 {\sqrt{2 \pi}} \int _{\mathbf{T}}e^{-ikx}f(x)\,dx,\qquad k\in \mathbf{Z},
\end{equation}
so that for suitable functions $f$ and $g$ on $\mathbf{T}$ we have
\begin{gather*}
f(x)=\frac{1}{\sqrt{2\pi}}\sum _{k\in \mathbf{Z}}\hat{f}(k)e^{ikx},\qquad \| f\| _{L^2(\mathbf{T})}^2 = \| \hat{f}\| _{\ell ^2(\mathbf{Z})}^2,\\
\widehat{\partial _xf}(k)=ik\hat{f}(k),\qquad \widehat{fg}(k)=\frac{1}{\sqrt{2\pi}}\sum _{l\in \mathbf{Z}}\hat{f}(k-l)\hat{g}(l).
\end{gather*}
The Sobolev norms are defined as $\| f\|_{H^s}:=\| \langle \,\cdot\,\rangle ^s\hat{f}\| _{\ell ^2(\mathbf{Z})}$, where $\langle \xi \rangle :=1+|\xi |$ for $\xi \in \mathbf{R}$.
We define the operator $P_{\pm}$ on $L^2(\mathbf{T})$ by
\[ P_\pm f(x):=\frac{1}{\sqrt{2\pi}}\sum _{k\in \mathbf{Z};\,\pm k>0}\hat{f}(k)e^{ikx}.\]
We denote by $X\lesssim Y$ the estimate $X\le CY$ with a harmless constant $C>0$.
Finally, we write $x+$ (resp.~$x-$) to denote a slightly bigger (resp.~smaller) number than a given $x\in \mathbf{R}$.

%
%

\section{Nonexistence of $H^s$ solutions}\label{sec:il}

In this section, we shall prove the following two theorems, which imply Theorem \ref{ip0} as a special case.

\begin{theorem} \label{ip} We assume that \eqref{h1} holds.
Let real numbers $s, s_1$ satisfy
\[ 1\le s_1\le s <s_1+1.\]
Then, there exists $u_0\in H^s (\mathbf{T})$ such that for \emph{no} $T>0$ the Cauchy problem \eqref{3NLSR} and \eqref{ic} has a solution $u\in C([0,T];H^{s_1} (\mathbf{T}))$ on $[0,T]$, or a solution $u\in C([-T,0];H^{s_1}(\mathbf{T}))$ on $[-T,0]$.
\end{theorem}

\begin{theorem}\label{ip+}
We assume that \eqref{h1} holds.
Let $s\ge 1$, and let $u\in C(I;H^{\frac{1}{2}+}(\mathbf{T}))$ be a solution to \eqref{3NLSR}--\eqref{ic} on a closed interval $I$ containing $0$ such that $u(t)\in H^s(\mathbf{T})$ for $t\in I$ and $\sup_{t\in I}\| u(t)\|_{H^s}<\infty$.
The following holds for any $T>0$.
\begin{enumerate}
\item If $I=[0,T]$, then $P_+u_0\in H^{s+\frac{1}{2}-}$.
\item If $I=[-T,0]$, then $P_-u_0\in H^{s+\frac{1}{2}-}$.
\item If $I=[-T,T]$, then $u(t)\in H^\infty$ for $t\in (-T,T)$ and $u_0$ satisfies
\[ \| u_0\|_{H^s}\le C_1^sR_1^{s^2} \]
for any $s\ge 1$ with the constants $C_1=C_1(\sup\limits_{t\in I}\| u(t)\|_{H^1},\,\| u_0\|_{L^2}^{-1},\,T^{-1})>0$ and $R_1>0$ independent of $s$.
\end{enumerate}
\end{theorem}

\begin{remark}\label{rem:ip+}
An immediate consequence of Theorem~\ref{ip+} is the following:
Let $1\le s_1\le s$ and $u_0\in H^s(\mathbf{T})$.
\begin{enumerate}
\item If ($s<s_1+\frac{1}{2}$ and) $P_+u_0\not\in H^{s_1+\frac{1}{2}-}$, then for \emph{no} $T>0$ there exists a solution $u\in C([0,T];H^{s_1}(\mathbf{T}))$ on $[0,T]$.
\item If ($s<s_1+\frac{1}{2}$ and) $P_-u_0\not\in H^{s_1+\frac{1}{2}-}$, then for \emph{no} $T>0$ there exists a solution $u\in C([-T,0];H^{s_1}(\mathbf{T}))$ on $[-T,0]$.
\item If $u_0\not\in H^\infty$, then for \emph{no} $T>0$ there exists a solution\\
 $u\in C([-T,T];H^{1}(\mathbf{T}))$ on $[-T,T]$.
\item If $u_0\in H^\infty$ and the estimate
\[ \sup _{s\ge 1}R^{-s^2}\| u_0\|_{H^s}<\infty \]
is \emph{false} for any $R>0$, then for \emph{no} $T>0$ there exists a solution $u\in C([-T,T];H^{1}(\mathbf{T}))$ on $[-T,T]$.
\end{enumerate}
\end{remark}

\begin{remark}
(a) Let us compare these theorems.

On one hand, when $1\le s_1\le s<s_1+\frac{1}{2}$, Theorem~\ref{ip+} (Remark~\ref{rem:ip+}) is stronger than Theorem~\ref{ip} in the sense that the nonexistence is shown for general initial data.
Moreover, Theorem~\ref{ip+} shows the nonexistence of solution on $[-T,T]$ for arbitrarily large $s$ (not necessarily satisfying $s<s_1+1$).

On the other hand, when $s_1+\frac{1}{2}\le s<s_1+1$, Theorem~\ref{ip} is stronger in that it shows the existence of an $H^s$ function which cannot be the initial data for a solution forward in time, nor a solution backward in time. 
Note that Theorem~\ref{ip+} does not exclude the possibility of the existence of a solution toward only one side from $t=0$.

(b) Theorem~\ref{ip+}~(iii) suggests that a solution may not exist even for $C^\infty$ initial data.
In fact, the function $u_0\in H^\infty$ defined by
\[ \hat{u}_0(k):=e^{-[\log \langle k\rangle ]^{4/3}},\qquad k\in \mathbf{Z}\]
satisfies $\| u_0\| _{H^s}\ge \langle e^{s^2}-1\rangle ^s|\hat{u}_0(e^{s^2}-1)|=e^{s^3-s^{8/3}}$ for any $s\ge 1$ such that $e^{s^2}\in \mathbf{Z}$, and by Remark~\ref{rem:ip+}~(iv) it cannot be the initial data for a solution in $C([-T,T];H^1)$.
\end{remark}

Before proving these theorems, we see the $L^2$ conservation for \eqref{3NLSR}.

\begin{lemma}\label{lem:L2law}
Let $T>0$, $s>1/2$, and $u\in C([0,T];H^s(\mathbf{T}))$ be a solution to the Cauchy problem \eqref{3NLSR}--\eqref{ic} on $[0,T]$ with initial data $u_0\in H^s(\mathbf{T})$.
Then, we have
\begin{equation*}
\| u(t)\| _{L^2(\mathbf{T})}=\| u_0\| _{L^2(\mathbf{T})},\qquad t\in [0,T].
\end{equation*}
\end{lemma}

\begin{proof}
It is not hard to derive the conservation law formally, so we only see how to make it rigorous under the assumption on regularity.

Let $P_{\le N}$ be the projection onto frequency range $\{ k\in \mathbf{Z}\,|\,|k|\le N\}$.
Then, $u_N:=P_{\le N}u$ solves 
\begin{align*}
\partial_t u_N  =&\; \alpha_1 \partial_x^3 u_N + i \alpha_2 \partial_x^2 u_N + P_{\le N}\bigl ( i\gamma_1 |u|^2 u +\gamma_2 \partial_x \bigl ( |u|^2 u \bigr ) - i\Gamma u \partial_x \bigl (|u|^2\bigr )\bigr ) \\
=&\; \alpha_1 \partial_x^3 u_N + i \alpha_2 \partial_x^2 u_N + i\gamma_1 |u_N|^2 u_N +\gamma_2 \partial_x ( |u_N|^2 u_N ) - i\Gamma u_N \partial_x (|u_N|^2)\\
& +i\gamma_1 \bigl ( P_{\le N}(|u|^2 u)-|u_N|^2u_N\bigr ) +\gamma_2 \bigl ( P_{\le N}\partial_x ( |u|^2 u) -\partial_x \bigl ( |u_N|^2 u_N \bigr ) \bigr ) \\
& - i\Gamma \bigl ( P_{\le N}\bigl ( u \partial_x (|u|^2)\bigr ) - u_N \partial_x \bigl (|u_N|^2\bigr )\bigr ) ,\\
u_N(0) &= P_{\le N}u_0.
\end{align*}
Note that $u_N$ is smooth and the above equality holds in the classical sense.
Taking the real $L^2$ inner product with $u_N$ and then integrating over $(0,t)$, we obtain after some integration by parts that
\begin{equation}\label{L2law-0}
\| u_N(t)\| _{L^2}^2= \| P_{\le N}u_0\| _{L^2}^2+\int _0^t R_N(t')\,dt',\qquad t\in (0,T],
\end{equation}
where
\begin{align*}
R_N(t):=&\; 2\Re \int _{\mathbf{T}} i\gamma_1 P_{\le N}\bigl [ |u|^2 u-|u_N|^2u_N\bigr ] (t)\bar{u}_N(t)\,dx\\
&+2\Re \int _{\mathbf{T}}\gamma_2 P_{\le N}\bigl [ \partial_x ( |u|^2 u) -\partial_x \bigl ( |u_N|^2 u_N \bigr )\bigr ] (t)\bar{u}_N(t)\,dx\\
&-2\Re \int _{\mathbf{T}}i\Gamma P_{\le N}\bigl [ u \partial_x (|u|^2)- u_N \partial_x \bigl (|u_N|^2\bigr )\bigr ] (t)\bar{u}_N(t)\,dx.
\end{align*}
Since the Sobolev estimate
\[ \| fg\partial_x h\| _{H^{-s}}\lesssim \| f\| _{H^s}\| g\| _{H^s}\| h\| _{H^s}\]
is valid if $s>1/2$, we have
\[ |R_N(t)|\lesssim \| u(t)\| _{H^s}^3\| u(t)-u_N(t)\| _{H^s}.\]
Finally, taking $N\to \infty$ in \eqref{L2law-0} and noticing that $\| u-u_N\| _{L^\infty (0,t;H^s)}\to 0$ as $N\to \infty$ for $u\in C([0,t];H^s)$, we obtain the desired $L^2$ conservation law.
\end{proof}

We will give proofs of Theorems~\ref{ip}, \ref{ip+} after some discussion on general $H^s$ solutions to the Cauchy problem.

Let $s\ge 1$, $u_0\in H^s(\mathbf{T})$, $T>0$, and let $u$ be a solution to the Cauchy problem on $[0,T]$ which belongs to $C([0,T];H^s(\mathbf{T}))$.
By Remark~\ref{remsol} (ii) (iii), $\hat{u}(\cdot ,k)\in C^1([0,T])$ and \eqref{hateq}--\eqref{hatic} holds for any $k\in \mathbf{Z}$.

We introduce a new function $v\in C([0,T];H^s(\mathbf{T}))$ by
\begin{equation*}
\hat{v}(t,k):=e^{i(\alpha_1 k^3+\alpha_2 k^2)t}\hat{u}(t,k),\qquad (t,k)\in [0,T]\times \mathbf{Z}.
\end{equation*}
Observe that $\hat{v}(\cdot ,k)\in C^1([0,T])$ and it is a solution to 
\begin{align}
\begin{split}
\partial _t\hat{v}(t,k)=& \sum _{k_1+k_2+k_3=k} \frac{i\gamma_1 +i\gamma_2 k+\Gamma (k_1+k_2)}{2\pi}e^{it\Phi}\hat{v}(t,k_1)\hat{\bar{v}}(t,k_2)\hat{v}(t,k_3),
\end{split}\label{eq:hat-v} \\
\hat{v}(0,k)=&\; \hat{u}_0(k)\label{ic:hat-v},
\end{align}
where 
\begin{equation}\label{def:Phi}
\begin{split}
\Phi =&\;\Phi (k_1,k_2,k_3)\\
:=&\;\bigl ( \alpha_1 (k_1+k_2+k_3)^3+\alpha_2 (k_1+k_2+k_3)^2\bigr ) \\
&-(\alpha_1 k_1^3+\alpha_2 k_1^2)+\bigl ( \alpha_1 (-k_2)^3+\alpha_2 (-k_2)^2\bigr ) -(\alpha_1 k_3^3+\alpha_2 k_3^2)\\
=&\;3\alpha_1 (k_1+k_2)(k_2+k_3)\bigl ( k_3+k_1+\frac{2\alpha_2}{3\alpha_1}\bigr ) .
\end{split}
\end{equation}
Under the assumption $\frac{2\alpha_2}{3\alpha_1}\not\in \mathbf{Z}$, it holds that
\begin{gather}
\Phi (k_1,k_2,k_3)=0\quad \Leftrightarrow \quad (k_1+k_2)(k_2+k_3)=0, \label{cond:Phi=0}\\ 
\label{est:Phi}
\Phi (k_1,k_2,k_3)\neq 0\quad \Rightarrow \quad |\Phi (k_1,k_2,k_3)|\sim \langle k_1+k_2\rangle \langle k_2+k_3\rangle  \langle k_3+k_1\rangle .
\end{gather}

As in the preceding section, we notice \eqref{cond:Phi=0} and separate the resonant terms ($\Phi =0$) from the summation in \eqref{eq:hat-v} to have
\begin{align*}
\partial _t\hat{v}(k)
=&\; \Bigl [ \sum _{\begin{smallmatrix} k_1+k_2+k_3=k\\ (k_1+k_2)(k_2+k_3)\neq 0\end{smallmatrix}} +\sum _{\begin{smallmatrix} k_1+k_2=0\\ k_3=k\end{smallmatrix}}+\sum _{\begin{smallmatrix} k_2+k_3=0\\ k_1=k\end{smallmatrix}}-\sum _{\begin{smallmatrix} k_1=-k_2=k_3=k\end{smallmatrix}}\Bigr ] \\
&\qquad \qquad \qquad \frac{i\gamma_1 +i\gamma_2 k+\Gamma (k_1+k_2)}{2\pi}e^{it\Phi}\hat{v}(k_1)\bar{\hat{v}}(-k_2)\hat{v}(k_3)\\
=&\; \frac{i\gamma_1+i\gamma_2k}{\pi}\bigl ( \sum _{k'\in \mathbf{Z}}|\hat{v}(k')|^2\bigr ) \hat{v}(k) +\frac{\Gamma}{2\pi}\bigl ( \sum _{k'\in \mathbf{Z}}(k-k')|\hat{v}(k')|^2\bigr ) \hat{v}(k)\\
& -\frac{i\gamma_1+i\gamma_2k}{2\pi}|\hat{v}(k)|^2\hat{v}(k) \\
& +\sum _{\begin{smallmatrix} k_1+k_2+k_3=k\\ (k_1+k_2)(k_2+k_3)\neq 0\end{smallmatrix}}\frac{i\gamma_1+i\gamma_2 k+\Gamma (k_1+k_2)}{2\pi}e^{it\Phi}\hat{v}(k_1)\bar{\hat{v}}(-k_2)\hat{v}(k_3).
\end{align*}

We further move to a reduced equation, as we did in \eqref{red3NLS}, by introducing a new function $w\in C([0,T];H^s(\mathbf{T}))$ as
\begin{equation*}
\hat{w}(t,k):=\exp \bigl ( -i\frac{\gamma_1+\gamma_2k}{\pi}\| u_0\| _{L^2(\mathbf{T})}^2t\bigr ) \hat{v}(t,k),\qquad (t,k)\in [0,T]\times \mathbf{Z}.
\end{equation*}
Recalling the $L^2$ conservation law established in Lemma~\ref{lem:L2law}, we observe that $\hat{w}$ solves
\begin{align}
\begin{split}
\partial _t\hat{w}(k)
=&\; \frac{\Gamma}{2\pi}\| u_0\| _{L^2(\mathbf{T})}^2k\hat{w}(k)\\
& -\Bigl [ \frac{\Gamma}{2\pi}\bigl ( \sum _{k'\in \mathbf{Z}}k'|\hat{w}(k')|^2\bigr ) \hat{w}(k)+\frac{i\gamma_1+i\gamma_2k}{2\pi}|\hat{w}(k)|^2\hat{w}(k) \Bigr ] \\
& +\frac{i\gamma_1}{2\pi}\sum _{\begin{smallmatrix} k_1+k_2+k_3=k\\ (k_1+k_2)(k_2+k_3)\neq 0\end{smallmatrix}}e^{it\Phi}\hat{w}(k_1)\bar{\hat{w}}(-k_2)\hat{w}(k_3)\\
& +\sum _{\begin{smallmatrix} k_1+k_2+k_3=k\\ (k_1+k_2)(k_2+k_3)\neq 0\end{smallmatrix}}\frac{i\gamma_2 k+\Gamma (k_1+k_2)}{2\pi}e^{it\Phi}\hat{w}(k_1)\bar{\hat{w}}(-k_2)\hat{w}(k_3)\\
=:&\; \frac{\Gamma}{2\pi} \| u_0\| _{L^2}^2k\hat{w}(k)+F_1[w(t)](k)+F_2[w(t)](k)+F_3[w(t)](k),
\end{split}\label{eq:hat-w} \\
\hat{w}(0,k)=&\; \hat{u}_0(k)\label{ic:hat-w}.
\end{align}

\begin{remark} \label{redeq}
The above transform into the reduced equation is not needed for showing Theorems~\ref{ip} and \ref{infl} as long as $s<s_1+\frac{1}{2}$.
Moreover, for Theorem~\ref{ip}, we need the resonant/nonresonant decomposition as above only for the nonlinear terms with derivative.
Nevertheless, we have derived the fully reduced equation for later use.
\end{remark}

The first term on the right-hand side of \eqref{eq:hat-w} causes exponential growth of the positive modes of $w(t)$, which is expected to make the Cauchy problem ill-posed.
In fact, we will see that $F_1$ and $F_2$ can be easily estimated, while one can control $F_3$ by an integration by parts in $t$, thanks to the nonresonant property.

Before applying an integration by parts to $F_3$, we decompose it into two parts as
\begin{align*}
F_3[w](k)=&\;\bigl ( \sum _{D_1(k)}+\sum _{D_2(k)}\bigr ) \frac{i\gamma_2 k+\Gamma (k_1+k_2)}{2\pi}e^{it\Phi}\hat{w}(k_1)\bar{\hat{w}}(-k_2)\hat{w}(k_3)\\
=:&\; F_{3,1}[w](k)+F_{3,2}[w](k),
\end{align*}
where
\begin{align*}
D_1(k):=&\; \{ (k_1,k_2,k_3)\in D(k)\,|\, \tfrac{1}{4}|k_2|\le |k_1|,|k_3|\le 4|k_2|\} ,\\
D_2(k):=&\; D(k)\setminus D_1(k),\\
D(k):=&\; \{ (k_1,k_2,k_3)\in \mathbf{Z}^3\,|\, k_1+k_2+k_3=k,\,(k_1+k_2)(k_2+k_3)\neq 0\} .
\end{align*}
We first note that
\begin{equation}
(k_1,k_2,k_3)\in D_1(k)\quad \Rightarrow \quad |k|\lesssim |k_1|\sim |k_2|\sim |k_3|. \label{est:D1}
\end{equation}
Moreover, if we recall \eqref{est:Phi}, then it is not hard to show that
\begin{equation}
(k_1,k_2,k_3)\in D_2(k)\quad \Rightarrow \quad |\Phi (k_1,k_2,k_3)|\gtrsim \max _{1\le j\le 3} \langle k_j \rangle ^2. \label{est:D2}
\end{equation} 
We now rewrite $F_{3,2}$ as
\begin{equation*}
\begin{split}
F_{3,2}[w(t)](k)=&\; \partial _tG[w(t)](k)+H[w(t)](k),\\
G[w(t)](k):=&\; \sum _{D_2(k)}\frac{i\gamma_2 k+\Gamma (k_1+k_2)}{2\pi i\Phi}e^{it\Phi}\hat{w}(t, k_1)\bar{\hat{w}}(t, -k_2)\hat{w}(t, k_3),\\
H[w(t)](k):=&\; -\sum _{D_2(k)}\frac{i\gamma_2 k+\Gamma (k_1+k_2)}{2\pi i\Phi}e^{it\Phi}\partial _t\bigl [ \hat{w}(t, k_1)\bar{\hat{w}}(t, -k_2)\hat{w}(t, k_3)\bigr ].
\end{split}
\end{equation*}
Here, we notice that absolute and uniform-in-$t$ convergence of the summation over $(k_1,k_2,k_3)$ allows us to exchange the order of summation and differentiation in $t$.

So far, we have obtained the following equation on $\hat{w}(t,k)$:
\begin{equation*}
\begin{split}
\partial _t\hat{w}(t,k)
=&\; \frac{\Gamma}{2\pi}  \| u_0\| _{L^2}^2k\hat{w}(t,k)+\partial _tG[w(t)](k)\\
&+F_1[w(t)](k)+F_2[w(t)](k)+F_{3,1}[w(t)](k)+H[w(t)](k).
\end{split}
\end{equation*}
Let us assume that $a:=\frac{\Gamma}{2\pi}\| u_0\| _{L^2(\mathbf{T})}^2 >0$, \mbox{i.e.}, $\| u_0\| _{L^2(\mathbf{T})}\neq 0$.
By \eqref{ic:hat-w} we have
\begin{align*}
\hat{w}(T,k)
&=e^{akT}\hat{w}(0,k)+\int _0^T e^{ak(T-t)}\Big( \partial _tG+F_1+F_2+F_{3,1}+H\Big) [w(t)](k)\,dt\\
&=e^{akT}\hat{u_0}(k)+\Big( G[w(T)](k)-e^{akT}G[u_0](k)\Big) \\
&\quad +\int _0^Te^{ak(T-t)}\big( akG+F_1+F_2+F_{3,1}+H\big) [w(t)](k)\,dt.
\end{align*}
In particular, it holds that
\begin{equation}\label{eq:hat-w2}
\begin{split}
\hat{u}_0(k)
&=e^{-akT}\hat{w}(T,k)-\Big( e^{-akT}G[w(T)](k)-G[u_0](k)\Big) \\
&\quad -\int _0^Te^{-akt}\big( akG+F_1+F_2+F_{3,1}+H\big) [w(t)](k)\,dt.
\end{split}
\end{equation}

\begin{lemma}\label{lem:multilinear}
There exists a constant $R>1$ such that for any $s\ge 1$ the following holds, with the implicit constants in all the estimates being independent on $s$.

(i) We have
\begin{gather}
\sup _{k\in \mathbf{Z}} \; \langle k\rangle ^s\bigl | F_1[f](k)+F_2[f](k)+F_{3,1}[f](k)\bigr |\lesssim R^s\| f\| _{H^{\frac{3}{4}}(\mathbf{T})}^2\| f\| _{H^s(\mathbf{T})},\label{est:F}\\
\sup _{k\in \mathbf{Z}} \; \langle k\rangle ^{s+1}\bigl | G[f](k)\bigr |\lesssim R^s\| f\| _{H^{\frac{1}{4}}(\mathbf{T})}^2\| f\| _{H^s(\mathbf{T})}.\label{est:G}
\end{gather}

(ii) Moreover, if $\hat{w}(\cdot ,k)\in C^1([0,T])$ is a solution to \eqref{eq:hat-w} and $w(t)\in H^s(\mathbf{T})$ for $t\in [0,T]$, then we have
\begin{gather}
\sup _{k\in \mathbf{Z}}\; \langle k\rangle ^{s-1}\bigl | \partial _t\hat{w}(t,k)\bigr |\lesssim R^s\| w(t)\| _{H^{\frac{1}{4}}(\mathbf{T})}^2\| w(t)\| _{H^s(\mathbf{T})},\label{est:w_t}\\
\sup _{k\in \mathbf{Z}}\; \langle k\rangle ^s\bigl | H[w(t)](k)\bigr | \lesssim R^s\| w(t)\| _{H^{\frac{1}{4}}(\mathbf{T})}^2\| w(t)\| _{H^{\frac{1}{2}+}(\mathbf{T})}^2\| w(t)\| _{H^s(\mathbf{T})}.\label{est:H}
\end{gather}
\end{lemma}

\begin{proof}
We begin with \eqref{est:F}.
$F_1$ is easily estimated as
\begin{equation*}
\langle k\rangle ^s\bigl | F_1[f](k)\bigr |\lesssim \bigl ( \| \langle k\rangle ^{\frac{1}{2}}\hat{f}\| _{\ell ^2}^2+\| \langle k\rangle ^{\frac{1}{2}}\hat{f}\| _{\ell ^\infty}^2\bigr ) \| \langle k\rangle ^s\hat{f}\| _{\ell ^\infty} \lesssim \| \langle k\rangle ^{\frac{1}{2}}\hat{f}\| _{\ell ^2}^2\| \langle k\rangle ^s\hat{f}\| _{\ell ^2}.
\end{equation*}
For estimate on $F_2$ we notice that $\langle k\rangle ^s\le R^s\big( \langle k_1\rangle ^s+\langle k_2\rangle ^s+\langle k_3\rangle ^s\big)$ and use Sobolev embedding to obtain that
\begin{equation*}
\langle k\rangle ^s\bigl | F_2[f](k)\bigr |\lesssim R^s\| \mathcal{F}^{-1}[|\hat{f}|]\| _{L^4(\mathbf{T})}^2\| f\|_{H^s(\mathbf{T})}
\lesssim R^s\| f\| _{H^{\frac{1}{4}}}^2\| f\| _{H^s}.
\end{equation*}
By \eqref{est:D1}, it holds that $\langle k\rangle ^s|i\gamma_2 k+\Gamma(k_1+k_2)|\lesssim R^s\langle k_1\rangle ^{\frac{1}{2}}\langle k_2\rangle ^{\frac{1}{2}}\langle k_3\rangle ^s$ in $D_1(k)$.
Hence $F_{3,1}$ is estimated by Sobolev embedding as
\begin{equation*}
\langle k\rangle ^s\bigl | F_{3,1}[f](k)\bigr |\lesssim R^s\| \mathcal{F}^{-1}[\langle k\rangle ^{\frac{1}{2}}|\hat{f}|]\| _{L^4}^2\| f\| _{H^s}\lesssim R^s\| f\| _{H^{\frac{3}{4}}}^2\| f\| _{H^s}.
\end{equation*}

Similarly to the estimate on $F_2$ above, \eqref{est:w_t} can be shown by the equation \eqref{eq:hat-w}, the H\"older inequality and Sobolev embedding.

To show \eqref{est:G} and \eqref{est:H} we recall \eqref{est:D2}, which implies that $\langle k\rangle |i\gamma_2 k+\Gamma(k_1+k_2)|\lesssim |\Phi (k_1,k_2,k_3)|$ in $D_2(k)$.
Then, proof of \eqref{est:G} is similar to the estimate on $F_2$ above.
Finally, by \eqref{est:w_t} we have
\begin{align*}
&\langle k\rangle ^s\bigl | H[w(t)](k)\bigr |\\
&\lesssim R^s\Big( \| \langle k\rangle ^{s-1}\hat{w}(t)\| _{\ell ^1}\| \hat{w}(t)\| _{\ell ^1}\| \partial _t\hat{w}(t)\| _{\ell ^\infty} +\| \hat{w}(t)\| _{\ell ^1}^2\| \langle k\rangle ^{s-1}\partial _t\hat{w}(t)\| _{\ell ^\infty} \Big) \\
&\lesssim R^s\Bigl ( \| w(t)\| _{H^{s-\frac{1}{2}+}}\| w(t)\| _{H^{\frac{1}{2}+}}\| w(t)\| _{H^1}+\| w(t)\| _{H^{\frac{1}{2}+}}^2\| w(t)\| _{H^s}\Bigr ) \| w(t)\| _{H^{\frac{1}{4}}}^2,
\end{align*}
and then \eqref{est:H} follows from interpolation.
\end{proof}

\begin{remark} \label{elliptic}
The nonresonance relation \eqref{est:D2} is used for the estimate \eqref{est:H} in Lemma \ref{lem:multilinear}, which is based on the dispersive nature of equation \eqref{eq:hat-w}.
It seems difficult to prove the ill-posedness simply by the elliptic regularity theorem of the Cauchy-Riemann operator in \eqref{eq:hat-w} without exploiting the dispersive nature.
Indeed, the point of the ill-posedness proof based on the elliptic regularity is to prove that a solution to \eqref{eq:hat-w} is in $C^\infty ((-T, T) \times \mathbf{T})$ for some $T > 0$, which is a contradiction to the fact that the initial data is in $H^s(\mathbf{T})$ and not in $C^\infty(\mathbf{T})$.
The Cauchy-Riemann operator is the elliptic operator of first order, so we can gain only one derivative over the inhomogeneous term.
Namely, if $u$ satisfies
\[
   (\partial_t - i \partial_x)u = f \ \textrm{in} \ \mathscr{D}'(\Omega), \quad f \in H^s(\Omega), 
\]
then for any $\Omega' \subset \subset \Omega$, we have $u \in H^{s+1}(\Omega')$,
where $\Omega$ is an open subset in $\mathbf{R} \times \mathbf{T}$.
But the second term $F_1$ and the fourth term $F_3$ on the right side of \eqref{eq:hat-w} contain the first derivative of the unknown function.
Therefore, without dispersion, we can not expect that when $w \in H^s((-T, T) \times \mathbf{T})$, the elliptic regularity property of equation \eqref{eq:hat-w} yields $w \in H^{s_0}((-T/2, T/2) \times \mathbf{T})$ for some $s_0 > s$.
Specifically, we need to use the nonresonance relation relevant to $\Phi$ (see \eqref{def:Phi}) for the estimate of the following partial sum appearing in $F_3$:
\[
   \sum_{\begin{subarray}{c} k_1 + k_2 + k_3 = k \\ (k_1 + k_2)(k_2 + k_3) \neq 0 \\ k \sim k_1, \ |k| \gg |k_2|, |k_3| \end{subarray}} k e^{it \Phi} \hat w(k_1) \bar{\hat w}(-k_2) \hat w (k_3).
\]
\end{remark}

Since we consider a solution $u$ in $C([0,T];H^s(\mathbf{T}))$, we have
\begin{equation*}
E^+_{s'}(T):=\sup _{t\in [0,T]}\| u(t)\|_{H^{s'}}<\infty ,\qquad s'\le s.
\end{equation*}
From \eqref{eq:hat-w2} and Lemma~\ref{lem:multilinear}, together with $|\hat{w}(t,k)|=|\hat{u}(t,k)|$, we see that 
\begin{equation*}
\begin{split}
&\langle k\rangle ^{s+1}|\hat{u}_0(k)|\\
&\le e^{-akT}\langle k\rangle ^{s+1}|\hat{w}(T,k)|+\langle k\rangle ^{s+1}\Big( |G[w(T)](k)|+|G[u_0](k)|\Big) \\
&\quad +\langle k\rangle ^{s+1}\sup _{t\in [0,T]}\Big| \big( akG+F_1+F_2+F_{3,1}+H\big) [w(t)](k)\Big| \int _0^Te^{-akt}\,dt\\
&\le e^{-akT}\langle k\rangle E_s^+(T)+CR^sE_1^+(T)^2E_s^+(T) \\
&\quad +C\langle k\rangle R^s\Big( aE_1^+(T)^2+E_1^+(T)^2+E_1^+(T)^4\Big) E_s^+(T)\int _0^Te^{-akt}\,dt.
\end{split}
\end{equation*}
Hence, for any $k>0$ we have
\begin{equation*}
\begin{split}
\langle k\rangle ^{s+1}|\hat{u}_0(k)|
&\le e^{-akT}\langle k\rangle E_s^+(T)\\
&\quad +CR^s\Big( E_1^+(T)^2+\frac{E_1^+(T)^2+E_1^+(T)^4}{\| u_0\| _{L^2}^2}\Big) E_s^+(T).
\end{split}
\end{equation*}
Finally, using
\begin{equation*}
\sup _{x>0}(1+x)e^{-\alpha x}\le \max \{ 1,\alpha ^{-1}\}\qquad (\alpha >0),
\end{equation*}
and the above estimate, we obtain
\begin{equation}\label{est:hat-w}
\begin{split}
&\langle k\rangle ^{s+1}|\hat{u}_0(k)|\\
&\lesssim \Big\{ 1+\frac{1}{\| u_0\|_{L^2}^2T}+R^s\Big( E_1^+(T)^2+\frac{E_1^+(T)^2+E_1^+(T)^4}{\| u_0\|_{L^2}^2}\Big) \Big\} E_s^+(T)\\
&\lesssim R^s\big\langle \|u_0\|_{L^2}^{-2}\big\rangle \big\langle T^{-1}\big\rangle \big\langle E_1^+(T)\big\rangle ^4E_s^+(T)
\end{split}
\end{equation}
for any $k>0$, where the implicit constant is independent of $s$, $T$ and $k$.

We are now in a position to prove Theorems~\ref{ip} and \ref{ip+}.

\begin{proof}[Proof of Theorem~\ref{ip}]
Let $s,s_1$ be such that $1\le s_1\le s<s_1+1$.
We take any $s_0\in (s,s_1+1)$ and choose initial data $u_0$ defined by
\begin{equation}\label{def:u_0}
\hat{u}_0(k):=\begin{cases}
\langle k \rangle ^{-s_0}\quad &\text{if $k=\pm 2^j$ for some $j\in \mathbf{N}$,}\\
0 &\text{otherwise,}
\end{cases}
\end{equation}
which is clearly in $H^s(\mathbf{T})$.
Suppose for contradiction that there existed positive time $T>0$ and a solution $u\in C([0,T];H^{s_1}(\mathbf{T}))$ to the Cauchy problem \eqref{3NLSR}--\eqref{ic} on $[0,T]$.
Since $E^+_{s_1}(T)<\infty$, \eqref{est:hat-w} would imply that 
\begin{equation*}
\sup _{k>0}~\langle k\rangle ^{s_1+1}|\hat{u}_0(k)|<\infty ,
\end{equation*}
which is a contradiction.
Therefore, the Cauchy problem with such initial data has no solution forward in time.
From Remark~\ref{remsol} (i) and the fact that $\hat{\bar{u}}_0(k)=\bar{\hat{u}}_0(-k)$, we obtain the corresponding result on backward solutions.
\end{proof}

\begin{proof}[Proof of Theorem~\ref{ip+}]
We see (i) by the estimate 
\begin{equation}\label{1+}
\| P_+u_0\| _{H^{s+\frac{1}{2}-}}\lesssim R^s\big\langle \|u_0\|_{L^2}^{-2}\big\rangle \big\langle T^{-1}\big\rangle \big\langle E_1^+(T)\big\rangle ^4E_s^+(T),
\end{equation}
which follows from \eqref{est:hat-w} and the Cauchy-Schwarz inequality.
Similar result holds for negative time:
If $u\in C([-T,0];H^{\frac{1}{2}+}(\mathbf{T}))$ is a solution to \eqref{3NLSR}--\eqref{ic} such that 
\begin{equation*}
u(t)\in H^s(\mathbf{T})\quad (t\in [-T,0]),\qquad E_s^-(T):=\sup_{t\in [-T,0]}\| u(t)\| _{H^s}<\infty ,
\end{equation*}
then
\begin{equation}\label{1-}
\| P_-u_0\| _{H^{s+\frac{1}{2}-}}\lesssim R^s\big\langle \|u_0\|_{L^2}^{-2}\big\rangle \big\langle T^{-1}\big\rangle \big\langle E_1^-(T)\big\rangle ^4E_s^-(T),
\end{equation}
which shows (ii).
In particular, if $u$ is a solution on $[-T,T]$, we have
\begin{equation}\label{1}
\| u_0\| _{H^{s+\frac{1}{2}-}}\lesssim R^s\big\langle \|u_0\|_{L^2}^{-2}\big\rangle \big\langle T^{-1}\big\rangle \big\langle E_1(T)\big\rangle ^4E_s(T),
\end{equation}
where $E_s(T):=\max \{ E_s^+(T),\,E_s^-(T)\}$.
This last estimate \eqref{1} can be iterated to show the $H^\infty$ regularity in the interior of the interval.

Moreover, \eqref{1} yields precise bounds on higher Sobolev norms of the initial data as follows.
In fact, \eqref{1} implies 
\begin{equation*}
\| u_0\| _{H^{s+\frac{1}{2}-}}\le C_0R^{s-1}\big\langle T^{-1}\big\rangle E_s(T),
\end{equation*}
where $C_0:=C\langle \|u_0\|_{L^2}^{-2}\rangle \langle E_1(T)\rangle^4$ does not depend on $s$. 
Since $\| u(t)\| _{L^2}$ is conserved, we use this estimate with $s=1+(n-k)(\frac{1}{2}-)$ at the $k$-th iteration for $1\le k\le n$ and obtain that for any $n\in \mathbf{N}$,
\begin{equation*}
\begin{split}
\| u_0\| _{H^{1+n(\frac{1}{2}-)}}&\le C_0R^{(n-1)(\frac{1}{2}-)}\big\langle (\tfrac{T}{n})^{-1}\big\rangle E_{1+(n-1)(\frac{1}{2}-)}(\tfrac{T}{n})\\
&\le \cdots \le \prod _{k=1}^n\Big[ C_0R^{(n-k)(\frac{1}{2}-)}\big\langle (\tfrac{T}{n})^{-1}\big\rangle \Big] 
\cdot E_1(T). 
\end{split}
\end{equation*}
As a consequence, we have (iii).
\end{proof}

We next consider the case that $\alpha_1 = 0$ and $\alpha_2 \neq 0$, that is, \eqref{3NLSR} is the Schr\"odinger equation with derivative nonlinearity.
We have the same ill-posedness result as Theorem \ref{ip}.

\begin{proposition}\label{prop:NLS}
Assume $\alpha_1= 0$ and $\alpha_2 \neq 0$.
Then Theorem \ref{ip} still holds.
\end{proposition}

\begin{remark}
It is likely that one can show analogues of Theorem~\ref{ip+} and Theorems~\ref{infl}, \ref{infl+} below for the $\alpha_1= 0$ case by the same idea with some elaborations.
We will not pursue these problems to avoid technical issues.
\end{remark}

\begin{proof}
For $s_1\ge 1$ and a solution $u\in C([0,T];H^{s_1}(\mathbf{T}))$ to \eqref{3NLSR}--\eqref{ic} on $[0,T]$,
the function $\hat{v}(t,k):=e^{i\alpha_2 k^2t}\hat{u}(t,k)$ satisfies the same Cauchy problem \eqref{eq:hat-v}--\eqref{ic:hat-v} but with
\[ \Phi (k_1,k_2,k_3)=2\alpha _2(k_1+k_2)(k_2+k_3)\]
instead of \eqref{def:Phi}.
In this case, the lower bound \eqref{est:D2} of $\Phi$ in $D_2(k)$ is replaced with
\[ |\Phi (k_1,k_2,k_3)|\gtrsim \max _{1\le j\le 3}\langle k_j\rangle ,\]
which is not sufficient to recover the derivative loss for the high-low interactions (such as $(k_1,k_2,k_3)\in D(k)$ with $|k_1|\gg |k_2|,|k_3|$).
To overcome this difficulty, we will exploit a suitable gauge transformation (cf.~Hayashi and Ozawa \cite{HO}; see Herr \cite{Herr} for the gauge transformation in the periodic case).

We use the following notation.
For a function $f:\mathbf{T}\to \mathbf{C}$, 
\begin{gather*}
P_0f:=\frac{1}{\sqrt{2\pi}}\hat{f}(0),\qquad P_{\neq 0}f(x):=f(x)-P_0f,\\
\partial^{-1}f(x):=\frac{1}{\sqrt{2\pi}}\sum _{k\in \mathbf{Z}\setminus \{ 0\}}\frac{1}{ik}\hat{f}(k)e^{ikx}=\frac{1}{2\pi}\int _{\mathbf{T}}\int _z^xP_{\neq 0}f(y)\,dy\,dz.
\end{gather*}
Note that $\partial _xP_{\neq 0}=\partial _x$ and $\partial _x\partial^{-1}=\partial^{-1}\partial_x=P_{\neq 0}$.
Define
\[ \mathcal{G}_\lambda (f)(x):=\exp \bigl ( \lambda \partial^{-1}(|f|^2)(x)\bigr ),\qquad U(t,x):=\mathcal{G}_\lambda (u(t))(x)u(t,x)\]
with
\[ \lambda :=\frac{2\gamma_2-i\Gamma}{2i\alpha_2}\in \mathbf{C}\setminus i\mathbf{R}.\]
Here, we notice that $\| \mathcal{G}_\lambda (u)\| _{L^\infty}\le \exp \bigl ( \Re \lambda \| u\| _{L^2}^2\bigr )$, and that
\[ u\in H^s(\mathbf{T}),\quad s\ge 1\qquad \Longrightarrow \qquad \mathcal{G}_\lambda (u)\in H^{s+1}(\mathbf{T}).\] 

A calculation shows that $U(t,x)$ (formally) solves
\begin{equation}
\partial _tU=i\alpha_2 \partial _x^2U+(2\gamma_2 -i\Gamma )P_0(|u|^2)\partial _xU-\gamma_2 uU\partial _x\bar{u}+\mathcal{R}(u,U) \label{eq:U}
\end{equation}
with
\begin{align*}
\mathcal{R}(u,U):=&\;\Bigl [ i\gamma_1 |u|^2+\lambda \Bigl \{ \frac{3}{2}\gamma_2 P_{\neq 0}(|u|^4)-i\alpha_2 \lambda \bigl ( P_{\neq 0}(|u|^2)\bigr ) ^2 \\
&\qquad -2i\alpha_2 \lambda P_0(|u|^2)P_{\neq 0}(|u|^2)+2\alpha_2 \Im P_0(\bar{u}\partial_x u)\Bigr \} \Bigr ] U.
\end{align*}
In fact, one can show in a similar way to the proof of Lemma~\ref{lem:L2law} that \eqref{eq:U} holds in $C([0,T];H^{s_1-2}(\mathbf{T}))$. 
The second term in the right-hand side of \eqref{eq:U} is peculiar to the periodic problem. 
In the real line case, this term does not appear and it is known that the Cauchy problem is well-posed (see Hayashi and Ozawa \cite{HO} and Chihara \cite{Ch1}).

Then, $\hat{W}(t,k):=\exp \bigl ( i(\alpha_2 k^2-\frac{\gamma_2}{\pi}\| u_0\| _{L^2}^2k)t\bigr ) \hat{U}(t,k)$ solves
\begin{equation}\label{eq:hat-W}
\begin{split}
&\partial _t\hat{W}(k)=\frac{\Gamma}{2\pi}\| u_0\| _{L^2}^2k\hat{W}(k)-\frac{i\gamma_2}{2\pi}\sum _{\begin{smallmatrix} k_1+k_2+k_3=k\\ (k_1+k_2)(k_2+k_3)\neq 0\end{smallmatrix}}e^{i\Phi t}k_2\hat{w}(k_1)\bar{\hat{w}}(-k_2)\hat{W}(k_3)\\
&+\frac{i\gamma_2}{2\pi}\Bigl [ \bigl ( \sum _{k'\in \mathbf{Z}}k'|\hat{w}(k')|^2\bigr )\hat{W}(k)+\bigl ( \sum _{k'\in \mathbf{Z}}k'\hat{W}(k')\bar{\hat{w}}(k')\bigr ) \hat{w}(k) + k|\hat{w}(k)|^2\hat{W}(k) \Bigr ]\\
&+\exp \bigl ( i(\alpha_2 k^2-\frac{\gamma_2}{\pi}\| u_0\| _{L^2}^2k)t\bigr ) \hat{\mathcal{R}}(u,U),
\end{split}
\end{equation}
where $\hat{w}(t,k):=\exp \bigl ( i(\alpha_2 k^2-\frac{\gamma_2}{\pi}\| u_0\| _{L^2}^2k)t\bigr ) \hat{u}(t,k)$.

The last two terms in the right-hand side of \eqref{eq:hat-W} can be easily estimated as $F_1$ and $F_2$ in \eqref{eq:hat-w}.
To treat the second term, we introduce the decomposition $D(k)=D_1'(k)\cup D_2'(k)$ which is different from the previous one:
\begin{align*}
D_1'(k):=&\; \{ (k_1,k_2,k_3)\in D(k)\,|\, |k_2|\le 4|k_1|~~\text{or}~~|k_2|\le 4|k_3|\} ,\\
D_2'(k):=&\; D(k)\setminus D_1'(k).
\end{align*}
By observing that
\begin{align*}
(k_1,k_2,k_3)\in D_1'(k)&\quad \Rightarrow \quad |k|\lesssim \max \{ |k_1|,\,|k_3|\} ,\\
(k_1,k_2,k_3)\in D_2'(k)&\quad \Rightarrow \quad |\Phi (k_1,k_2,k_3)|\sim \langle k_2 \rangle ^2= \max _{1\le j\le 3} \langle k_j \rangle ^2,
\end{align*}
we can show analogous estimates to \eqref{est:F}--\eqref{est:H} and obtain an estimate corresponding to \eqref{est:hat-w}:
\begin{align*}
\sup _{k>0}~\langle k\rangle ^{s_1+1}|\hat{U}(0,k)|\le C,
\end{align*}
where the constant $C>0$ depends on the parameters in the equation, $s_1$, $\| u_0\| _{L^2}^{-1}$, $T^{-1}$, and $\sup _{t\in [0,T]}\| (u,U,w,W)(t)\| _{H^{s_1}}$.
Note that $\| w\| _{H^{s_1}}=\| u\| _{H^{s_1}}$, $\| W\| _{H^{s_1}}=\| U\| _{H^{s_1}}$ and $\| U\| _{H^{s_1}}$ can be estimated in terms of $\| u\| _{H^{s_1}}$.

Finally, we take initial data $\epsilon u_0\in H^s(\mathbf{T})$ with $u_0$ as in \eqref{def:u_0} and $0<\epsilon \ll 1$.
Note that $\| \lambda \partial ^{-1}(|\epsilon u_0|^2)\| _{L^\infty}\ll 1$ and
\[ \mathrm{supp}\;\hat{u}_0\cap \{ k>0\} ~=~\{ 1,2,4,8,\dots \} .\]
Then, since $\mathcal{G}=\mathcal{G}_\lambda (\epsilon u_0)\in H^{s+1}$, we have for any $k\in \mathrm{supp}\;\hat{u}_0\cap \{ k>0\}$,
\begin{align*}
\big| \widehat{\mathcal{G}u_0}(k)-\frac{1}{\sqrt{2\pi}} \hat{\mathcal{G}}(0)\hat{u}_0(k)\big| 
\le &\; \frac{1}{\sqrt{2\pi}} \sum _{|l|\ge k/2}|\hat{\mathcal{G}}(l)||\hat{u}_0(k-l)|\\
\lesssim &\; \langle k \rangle ^{-s-1}\| \mathcal{G}\| _{H^{s+1}}\| u_0\| _{L^2} \lesssim \langle k \rangle ^{-s-1}.
\end{align*}
The condition $\| \lambda \partial ^{-1}(|\epsilon u_0|^2)\| _{L^\infty}\ll 1$ implies that $|\frac{1}{\sqrt{2\pi}} \hat{\mathcal{G}}(0)-1|\ll 1$, and thus $U(0)=\mathcal{G}\epsilon u_0$ satisfies 
\[ |\hat{U}(0,k)| \ge \frac{\epsilon}{2}\langle k \rangle ^{-s_0}\]
for all sufficiently large $k=2^j$.
Therefore, we can argue as in the proof of Theorem~\ref{ip} to conclude that there exists no solution forward in time.
\end{proof}

%
%

\section{Norm inflation}\label{sec:infl}

Our results on norm inflation can be stated as the following two theorems, which imply Theorem~\ref{infl0} in particular.

\begin{theorem}\label{infl}
We assume that \eqref{h1} holds.
Let real numbers $s,s_1$ satisfy
\begin{equation}\label{cond:infl}
1\le s_1\le s<\min \{ \tfrac{5}{3}s_1+\tfrac{1}{2},\,s_1+\tfrac{3}{2} \} .
\end{equation}
Then, for any $\varepsilon ,\tau >0$ there exists a real analytic function $\psi_{\varepsilon ,\tau}$ satisfying $\|\psi_{\varepsilon ,\tau}\|_{H^s}\le \varepsilon$ such that if there exists a solution $u\in C([0,\tau ];H^{s_1}(\mathbf{T}))$ to \eqref{3NLSR} with the initial condition $u(0)=\psi_{\varepsilon,\tau}$, it holds that 
\[ \sup_{t\in [0,\tau ]}\| u(t)\| _{H^{s_1}}\ge \varepsilon ^{-1}.\]
The same is true for the negative time direction.
\end{theorem}

\begin{theorem}\label{infl+}
We assume that \eqref{h1} holds.

(i) Let $s\ge 1$.
Then, for any $\varepsilon ,\tau >0$ there exists a real analytic function $\phi_{\varepsilon ,\tau}$ satisfying $\|\phi_{\varepsilon ,\tau}\|_{H^s}\le \varepsilon$ such that the following holds:
Let $T>0$ and $u^*\in C([-T,T];H^1(\mathbf{T}))$ be a solution to \eqref{3NLSR} on $[-T,T]$.
Let $\varepsilon,\tau >0$ be such that $0<\tau \le T$, $\sup_{t\in [-\tau ,\tau ]}\| u^*(t)\|_{H^1}\le \varepsilon ^{-1}$, and such that $2\varepsilon \le \| u^*(0)\|_{L^2}$ if $u^*(0)\neq 0$.
Then, if there exists a solution $u\in C([-\tau,\tau ];H^1(\mathbf{T}))$ to \eqref{3NLSR} with the initial condition $u(0)=u^*(0)+\phi_{\varepsilon,\tau}$, it holds that 
\[ \sup_{t\in [-\tau ,\tau ]}\| u(t)-u^*(t)\| _{H^1}\ge \varepsilon ^{-1}.\]

(ii) Let $1\le s_1\le s<s_1+1$.
Then, for any $\varepsilon ,\tau >0$ there exists a real analytic function $\tilde\phi_{\varepsilon ,\tau}$ satisfying $\|\tilde\phi_{\varepsilon ,\tau}\|_{H^s}\le \varepsilon$ such that the following holds:
Let $T>0$ and $u^*\in C([0,T];H^1(\mathbf{T}))$ be a solution to \eqref{3NLSR} on $[0,T]$.
Let $\varepsilon,\tau >0$ be such that $0<\tau \le T$, $\sup_{t\in [0,\tau ]}\| u^*(t)\|_{H^{s_1}}\le \varepsilon ^{-1}$, and such that $2\varepsilon \le \| u^*(0)\|_{L^2}$ if $u^*(0)\neq 0$.
Then, if there exists a solution $u\in C([0,\tau ];H^{s_1}(\mathbf{T}))$ to \eqref{3NLSR} with the initial condition $u(0)=u^*(0)+\tilde\phi_{\varepsilon,\tau}$, it holds that 
\[ \sup_{t\in [0,\tau ]}\| u(t)-u^*(t)\| _{H^{s_1}}\ge \varepsilon ^{-1}.\]
The same is true for the negative time direction.
\end{theorem}

\begin{remark}
(a) For $s,s_1$ satisfying $1\le s_1\le s<s_1+1$, Theorem~\ref{infl+} (ii) is stronger than Theorem~\ref{infl+} (i) in that we do not need to assume the existence of solution to both sides from $t=0$, and also stronger than Theorem~\ref{infl} in that norm inflation occurs not only around $u^*\equiv 0$ but also around any solution $u^*$.

(b) We do not try to optimize the condition \eqref{cond:infl} in Theorem~\ref{infl}; our goal here is to include the case $s_1 = s-1$, which seems to be natural since the nonlinearity of \eqref{3NLSR} includes the first derivative of the unknown function.
However, in contrast to Theorem~\ref{infl+}, the estimate \eqref{est:hat-w'2} for each solution derived below cannot yield norm inflation at non-zero initial data.
See Remark~\ref{rem:infl} below.

(c) The analytic perturbations $\phi _{\varepsilon ,\tau},\tilde\phi _{\varepsilon ,\tau}$ which we take in the proof of Theorem~\ref{infl+} is depending only on $s,s_1,\varepsilon, \tau$ and independent of $u^*$.
\end{remark}

Let us begin with showing Theorem~\ref{infl+}, which is a direct consequence of the argument in the preceding section.

\begin{proof}[Proof of Theorem~\ref{infl+}]
(i) We can show the following estimate by iterating \eqref{1}:
For any $s'\!\ge 1$ there exist constants $C,C'\!>0$ such that for any solution $u\in C([-T,T];H^1)$ to the Cauchy problem \eqref{3NLSR}--\eqref{ic} on $[-T,T]$, it holds that
\begin{equation}\label{11}
\| u_0\| _{H^{s'}}\le C\Big[ \big\langle \|u_0\|_{L^2}^{-2}\big\rangle \big\langle T^{-1}\big\rangle \big\langle E_1(T)\big\rangle ^4\Big] ^{C'}E_1(T).
\end{equation}
(Hereafter, we do not care about the $s'$-dependence of the constants.)

We set
\[ \phi_{\varepsilon ,\tau}(x):=\frac{\varepsilon}{\sqrt{4\pi}}\big( 1+\langle k_0\rangle ^{-s}e^{ik_0x}\big) ,\]
where $k_0\in \mathbf{Z}$ is some large positive frequency to be chosen later.
Note that $\| \phi_{\varepsilon ,\tau}\| _{H^s}=\varepsilon$ and $\| \phi_{\varepsilon ,\tau}\| _{L^2}\sim \varepsilon$ are small, while $\| \phi_{\varepsilon ,\tau}\| _{H^{s+1}}\sim \varepsilon \langle k_0\rangle$ can be large.

We first consider the case $u^*(0)=0$ (we do not assume $u^*\equiv 0$).
Let $0<\tau \le T$, $\varepsilon ^{-1}\ge E_1^*(\tau ):=\sup_{t\in [-\tau ,\tau ]}\| u^*(t)\|_{H^1}$, and assume that a solution $u\in C([-\tau,\tau ];H^1)$ to \eqref{3NLSR} with $u(0)=\phi_{\varepsilon,\tau}$ exists.
Then, from \eqref{11} with $s'=s+1$ we have
\[ \varepsilon \langle k_0\rangle \le C\| u(0)\| _{H^{s+1}}\le C\Big[ \big\langle \varepsilon ^{-2}\big\rangle \big\langle \tau ^{-1}\big\rangle \big\langle E_1(\tau )\big\rangle ^4\Big] ^{C'}E_1(\tau ).\]
If we take $k_0$ as
\[ \varepsilon \langle k_0\rangle \ge C\Big[ \big\langle \varepsilon ^{-2}\big\rangle \big\langle \tau ^{-1}\big\rangle \big\langle 2\varepsilon ^{-1}\big\rangle ^4\Big] ^{C'}2\varepsilon ^{-1},\]
then it must hold that $E_1(\tau )\ge 2\varepsilon ^{-1}$.
Since $E_1^*(\tau )\le \varepsilon ^{-1}$, we have
\[ \sup _{t\in [-\tau ,\tau ]}\| u(t)-u^*(t)\| _{H^1}\ge E_1(\tau )-E_1^*(\tau )\ge \varepsilon ^{-1},\]
and the claim follows.

For $u^*(0)\neq 0$, by the assumption on $\varepsilon$ we have $\| u^*(0)+\phi_{\varepsilon ,\tau}\|_{L^2}\ge \varepsilon$.
We estimate $\phi_{\varepsilon,\tau}=(u^*(0)+\phi _{\varepsilon,\tau})-u^*(0)$ by \eqref{eq:hat-w2} of $u$ and $u^*$ to obtain
\begin{align*}
\| \phi _{\varepsilon ,\tau}\| _{H^{s+1}}&\lesssim \big\langle \| u^*(0)+\phi_{\varepsilon ,\tau}\|_{L^2}^{-2}\big\rangle \big\langle \tau ^{-1}\big\rangle \big\langle E_1(\tau )\big\rangle ^4E_{s+\frac{1}{2}+}(\tau )\\
&\quad +\big\langle \| u^*(0)\|_{L^2}^{-2}\big\rangle \big\langle \tau ^{-1}\big\rangle \big\langle E_1^*(\tau )\big\rangle ^4E_{s+\frac{1}{2}+}^*(\tau ).
\end{align*}
In the same manner as above, this inequality and \eqref{11} yield
\begin{align*}
\varepsilon \langle k_0\rangle &\le C\Big[ \big\langle \| u^*(0)+\phi_{\varepsilon ,\tau}\|_{L^2}^{-2}\big\rangle \big\langle \tau ^{-1}\big\rangle \big\langle E_1(\tau )\big\rangle ^4\Big] ^{C'}E_1(\tau )\\
&\quad +C\Big[ \big\langle \| u^*(0)\|_{L^2}^{-2}\big\rangle \big\langle \tau ^{-1}\big\rangle \big\langle E_1^*(\tau )\big\rangle ^4\Big] ^{C'}E_1^*(\tau )\\
&\le C\Big[ \big\langle \varepsilon ^{-2}\big\rangle \big\langle \tau ^{-1}\big\rangle \big\langle E_1(\tau )\big\rangle ^4\Big] ^{C'}E_1(\tau )+C\Big[ \big\langle (2\varepsilon )^{-2}\big\rangle \big\langle \tau ^{-1}\big\rangle \big\langle \varepsilon^{-1}\big\rangle ^4\Big] ^{C'}\varepsilon ^{-1}.
\end{align*}
Therefore, if we retake $k_0$ as 
\[ \varepsilon \langle k_0\rangle \ge C\Big[ \big\langle \varepsilon ^{-2}\big\rangle \big\langle \tau ^{-1}\big\rangle \big\langle 2\varepsilon ^{-1}\big\rangle ^4\Big] ^{C'}2\varepsilon ^{-1}+C\Big[ \big\langle (2\varepsilon )^{-2}\big\rangle \big\langle \tau ^{-1}\big\rangle \big\langle \varepsilon^{-1}\big\rangle ^4\Big] ^{C'}\varepsilon ^{-1},\]
it must hold that $E_1(\tau )\ge 2\varepsilon ^{-1}$.
The claim follows as before.

(ii) 
As $\tilde\phi_{\varepsilon,\tau}$, we use the same function $\phi_{\varepsilon,\tau}$ as above, but with a different $k_0$ depending on $s$ and $s_1$.
Note that $\| \tilde\phi_{\varepsilon ,\tau}\| _{H^s}=\varepsilon$, $\| \tilde\phi_{\varepsilon ,\tau}\| _{L^2}\sim \varepsilon$, and $\langle k_0\rangle ^{s_1+1}|\hat{\tilde\phi}_{\varepsilon ,\tau}(k_0)|\sim \varepsilon \langle k_0\rangle ^{s_1+1-s}$ can be large if $s<s_1+1$.
Then, the claim is shown in the same way as (i) using the estimate \eqref{est:hat-w} with $k=k_0$ and $s$ replaced by $s_1$, instead of \eqref{11}.
\end{proof}

We turn to the proof of Theorem~\ref{infl}, concentrating on the case of positive time direction as before.
One can actually show a slightly stronger estimate than \eqref{est:hat-w}: 
\begin{equation*}
\begin{split}
&\langle k\rangle ^{s_1+1}|\hat{u}_0(k)|\\
&\lesssim \Big( e^{-\frac{\Gamma}{2\pi}\| u_0\| _{L^2}^2kT}\langle k\rangle +E_{\frac{1}{4}}^+(T)^2+\frac{E_{\frac{3}{4}}^+(T)^2+E_{\frac{1}{4}}^+(T)^2E_{\frac{1}{2}+}^+(T)^2}{\| u_0\|_{L^2}^2}\Big) E_{s_1}^+(T)
\end{split}
\end{equation*}
for a solution $u\in C([0,T];H^{s_1})$ and any positive $k$, but it seems still useless for $s\ge s_1+1$.
In fact, the left-hand side of the above estimate is bounded with respect to $k>0$ and $u_0$ in a bounded set of $H^s$ when $s\ge s_1+1$.
Then, since $E_{3/4}^+(T)^2/\| u_0\| _{L^2}^2\ge 1$, this estimate is not likely to give some diverging lower bounds on $E_{s_1}^+(T)$ for a bounded sequence of initial data in $H^s$.
Hence, we need some refined estimates for proving Theorem~\ref{infl} when $s\ge s_1+1$.

Now, assume that $u\in C([0,T];H^{s_1}(\mathbf{T}))$ is a solution to \eqref{3NLSR}--\eqref{ic}, $s_1\ge 1$, $T>0$, and rewrite the equations \eqref{eq:hat-w}, \eqref{eq:hat-w2} for $\hat{w}(t,k)$ as follows:
\begin{align}
\partial _t\hat{w}(k)
=&\; \frac{\Gamma}{2\pi}\| u_0\| _{L^2}^2k\hat{w}(k) -\frac{\Gamma}{2\pi}\bigl ( \sum _{k'\in \mathbf{Z}}k'|\hat{w}(k')|^2\bigr ) \hat{w}(k)-\frac{i\gamma_1+i\gamma_2k}{2\pi}|\hat{w}(k)|^2\hat{w}(k) \notag \\
& +\sum _{\begin{smallmatrix} k_1+k_2+k_3=k\\ (k_1+k_2)(k_2+k_3)\neq 0\end{smallmatrix}}\frac{i\gamma_1+i\gamma_2 k+\Gamma (k_1+k_2)}{2\pi}e^{it\Phi}\hat{w}(k_1)\bar{\hat{w}}(-k_2)\hat{w}(k_3) \notag \\
=&\; \frac{\Gamma}{2\pi}\Bigl ( \| u_0\| _{L^2}^2k-P[u(t)]\Bigr ) \hat{w}(k) \notag \\
& +\Bigl [ \sum _{D_1(k)}\frac{i\gamma_1+i\gamma_2 k+\Gamma (k_1+k_2)}{2\pi}e^{it\Phi}\hat{w}(k_1)\bar{\hat{w}}(-k_2)\hat{w}(k_3) \notag \\
&\qquad -\frac{i\gamma_1+i\gamma_2k}{2\pi}|\hat{w}(k)|^2\hat{w}(k) \Bigr ] \notag \\
& +\partial _t \sum _{D_2(k)}\frac{i\gamma_1+i\gamma_2 k+\Gamma (k_1+k_2)}{2\pi i\Phi}e^{it\Phi}\hat{w}(k_1)\bar{\hat{w}}(-k_2)\hat{w}(k_3) \notag \\
& -\sum _{D_2(k)}\frac{i\gamma_1+i\gamma_2 k+\Gamma (k_1+k_2)}{2\pi i\Phi}e^{it\Phi}\partial_t \bigl [ \hat{w}(k_1)\bar{\hat{w}}(-k_2)\hat{w}(k_3)\bigr ] \notag \\
=:&\; \frac{\Gamma}{2\pi}\Bigl ( \| u_0\| _{L^2}^2k-P[u(t)]\Bigr ) \hat{w}(k) \notag \\
&+\tilde{F}[w(t)](k)+\partial _t\tilde{G}[w(t)](k)+\tilde{H}[w(t)](k), \notag
\end{align}
where
\[ P[u(t)]:=\Im \int _{\mathbf{T}}\bar{u}(t,x)\partial_x u(t,x)\,dx=\sum _{k\in \mathbf{Z}}k|\hat{u}(t,k)|^2=\sum _{k\in \mathbf{Z}}k|\hat{w}(t,k)|^2\in \mathbf{R}.\]
The quantity $P[u]$ is called the momentum.
Note that 
\[ |P[u(t)]|\le E_{\frac{1}{2}}^+(T)^2:=\sup _{t\in [0,T]}\| u(t)\|_{H^{\frac{1}{2}}}^2,\qquad t\in [0,T].\]
We consider one more decomposition; for $j=1,2,3$, let 
\[ D_{2,j}(k):=\bigl \{ (k_1,k_2,k_3)\in D_2(k)\,\big| \,|k_j|>4\max _{1\le l\le 3,\, l\neq j}|k_l|\bigr \} ,\]
and define 
\begin{align*}
\tilde{H}_1[w(t)](k):=&\; -\sum _{D_{2,1}(k)}\frac{i\gamma_1+i\gamma_2 k+\Gamma (k_1+k_2)}{2\pi i\Phi}e^{it\Phi}(\partial_t \hat{w})(k_1)\bar{\hat{w}}(-k_2)\hat{w}(k_3) \\
&\; -\sum _{D_{2,2}(k)}\frac{i\gamma_1+i\gamma_2 k+\Gamma (k_1+k_2)}{2\pi i\Phi}e^{it\Phi}\hat{w}(k_1)(\partial_t \bar{\hat{w}})(-k_2)\hat{w}(k_3) \\
&\; -\sum _{D_{2,3}(k)}\frac{i\gamma_1+i\gamma_2 k+\Gamma (k_1+k_2)}{2\pi i\Phi}e^{it\Phi}\hat{w}(k_1)\bar{\hat{w}}(-k_2)(\partial_t \hat{w})(k_3) ,\\
\tilde{H}_2[w(t)](k):=&\; \tilde{H}[w(t)](k)-\tilde{H}_1[w(t)](k). 
\end{align*}
Hence, we obtain 
\[ \partial _t\hat{w}(k)=\frac{\Gamma}{2\pi}\Bigl ( \| u_0\| _{L^2}^2k-P[u(t)]\Bigr ) \hat{w}(k) +\partial _t\tilde{G}[w(t)](k)+\big( \tilde{F}+\tilde{H}_1+\tilde{H}_2\big) [w(t)](k)
\]
for $t\in [0,T]$, which combined with \eqref{ic:hat-w} implies
\begin{equation}\label{eq:hat-w3}
\begin{split}
\hat{u}_0(k)
&=e^{-\frac{\Gamma}{2\pi}\int _0^T( \| u_0\| _{L^2}^2k-P[u(t)]) \,dt}\hat{w}(T,k)\\
&\quad -\int _0^{T}e^{-\frac{\Gamma}{2\pi}\int _0^t(\| u_0\|_{L^2}^2k-P[u(t')])\,dt'}\partial _t\tilde{G}[w(t)](k)\,dt\\
&\quad -\int _0^Te^{-\frac{\Gamma}{2\pi}\int _0^t( \| u_0\| _{L^2}^2k-P[u(t')]) \,dt'}\big( \tilde{F}+\tilde{H}_1+\tilde{H}_2\big) [w(t)](k)\,dt.
\end{split}
\end{equation}

Now, we assume that the solution $u$ and a positive integer $k$ satisfy
\begin{equation}\label{cond:k}
u_0\neq 0,\qquad 2E_{\frac{1}{2}}^+(T)^2\le \| u_0\| _{L^2}^2k.
\end{equation}
Then, we have
\[ \frac{\Gamma}{2\pi}\big( \| u_0\| _{L^2}^2k-P[u(t)])\ge \frac{\Gamma}{4\pi}\| u_0\|_{L^2}^2k>0, \qquad t\in [0,T].\]
We estimate the integral in the right-hand side of \eqref{eq:hat-w3} similarly to the estimates in \eqref{est:hat-w}.
First, by an integration by parts, we have
\begin{align*}
&\Big| \int _0^{T}e^{-\frac{\Gamma}{2\pi}\int _0^t(\| u_0\|_{L^2}^2k-P[u(t')])\,dt'}\partial _t\tilde{G}[w(t)](k)\,dt\Big| \\
&\le \Big| e^{-\frac{\Gamma}{2\pi}\int _0^T( \| u_0\| _{L^2}^2k-P[u(t)]) \,dt}\tilde{G}[w(T)](k)-\tilde{G}[u_0](k)\Big| \\
&\quad +\Big| \int _0^T\tfrac{\Gamma}{2\pi}\big( \| u_0\| _{L^2}^2k-P[u(t)]\big) e^{-\frac{\Gamma}{2\pi}\int _0^t( \| u_0\| _{L^2}^2k-P[u(t')]) \,dt'}\tilde{G} [w(t)](k)\,dt\Big| \\
&\le 4\sup _{t\in [0,T]}|\tilde{G}[w(t)](k)|.
\end{align*}
Secondly, we have
\begin{align*}
&\Big| \int _0^Te^{-\frac{\Gamma}{2\pi}\int _0^t( \| u_0\| _{L^2}^2k-P[u(t')]) \,dt'}\big( \tilde{F}+\tilde{H}_1+\tilde{H}_2\big) [w(t)](k)\,dt\Big| \\
&\le \sup _{t\in [0,T]}\big| \big( \tilde{F}+\tilde{H}_1+\tilde{H}_2\big) [w(t)](k)\big| \int _0^Te^{-\frac{\Gamma}{4\pi}\| u_0\|_{L^2}^2kt}\,dt \\
&\lesssim \frac{1}{\| u_0\|_{L^2}^2\langle k\rangle}\sup _{t\in [0,T]}\big| \big( \tilde{F}+\tilde{H}_1+\tilde{H}_2\big) [w(t)](k)\big| .
\end{align*}
Combining these estimates with \eqref{eq:hat-w3}, we obtain
\begin{equation}\label{est:hat-w'}
\begin{split}
|\hat{u}_0(k)|\lesssim \Big[ &\;e^{-\frac{\Gamma}{4\pi}\| u_0\| _{L^2}^2kT}|\hat{w}(T,k)|+\sup _{t\in [0,T]}|\tilde{G}[w(t)](k)|\\
&+\frac{1}{\| u_0\|_{L^2}^2\langle k\rangle}\sup _{t\in [0,T]}\big| \big( \tilde{F}+\tilde{H}_1+\tilde{H}_2\big) [w(t)](k)\big| \Big] 
\end{split}
\end{equation}
for any $k$ satisfying \eqref{cond:k}.

The next lemma is a refinement of Lemma~\ref{lem:multilinear}.
(Note that we do not clarify here the $s$-dependence of the constants.)
\begin{lemma}\label{lem:multilinear'}
Let $s\ge 1$.
We have
\begin{gather}
\sup _{k\in \mathbf{Z}} \; \langle k\rangle ^{3s-\frac{3}{2}}\bigl | \tilde{F}[f](k)\bigr |\lesssim \| f\| _{H^s}^3,\label{est:F'}\\
\sup _{k\in \mathbf{Z}} \; \langle k\rangle ^{s+1}\bigl | \tilde{G}[f](k)\bigr |\lesssim \| f\| _{H^{\frac{1}{4}}}^2\| f\| _{H^s}.\label{est:G'}
\end{gather}
Moreover, if $\hat{w}(\cdot ,k)\in C^1([0,T])$ is a solution to \eqref{eq:hat-w} and $w(t)\in H^s(\mathbf{T})$ for $t\in [0,T]$, then 
\begin{gather}
\sup _{k\in \mathbf{Z}}\; \langle k\rangle ^s\bigl | \tilde{H}_1[w(t)](k)\bigr | \lesssim \| w(t)\| _{H^{0+}}^2\| w(t)\| _{H^{\frac{1}{4}}}^2\| w(t)\| _{H^s},\label{est:H'1} \\
\sup _{k\in \mathbf{Z}}\; \langle k\rangle ^{s+1}\bigl | \tilde{H}_2[w(t)](k)\bigr | \lesssim \| w(t)\| _{H^{\frac{1}{4}}}^2\| w(t)\| _{H^{\frac{1}{2}+}}\| w(t)\| _{H^1}\| w(t)\| _{H^s}.\label{est:H'2}
\end{gather}
\end{lemma}

\begin{proof}
For \eqref{est:F'}, we observe that all the frequencies of three functions in each term of $F'$ are of the same size.
By the H\"older inequality and Sobolev embedding, we have
\begin{equation*}
\langle k\rangle ^{3s-\frac{3}{2}}\bigl | \tilde{F}[f](k)\bigr |\lesssim \| \mathcal{F}^{-1}[\langle k\rangle ^{s-\frac{1}{6}}|\hat{f}|]\| _{L^3}^3+\| \langle k\rangle ^{s-\frac{1}{6}}\hat{f}\| _{\ell ^\infty}^3\lesssim \| f\| _{H^s}^3.
\end{equation*}

The proof of \eqref{est:G'} is exactly the same as that of \eqref{est:G}.

To show \eqref{est:H'1} and \eqref{est:H'2}, we exploit the estimate
\[ | \Phi (k_1,k_2,k_3)| ^{-1}\lesssim \bigl [ \max_{1\le j\le 3}\langle k_j \rangle \bigr ] ^{-2}\Bigl ( \langle k_1+k_2 \rangle ^{-1}+\langle k_1+k_3 \rangle ^{-1}+\langle k_2+k_3 \rangle ^{-1}\Bigr ) ,\]
which is valid in $D_2(k)$ and improves \eqref{est:D2}.

A similar argument to the previous one for \eqref{est:H} then reduces the proof of \eqref{est:H'1} to showing
\begin{equation}\label{est:H'a}
\| \sum _{k_1+k_2+k_3=k}\frac{f_1(k_1)f_2(k_2)f_3(k_3)}{\langle k_1+k_2\rangle}\| _{\ell ^\infty}\lesssim \| f_1\| _{\ell ^2}\| \langle k\rangle ^{0+}f_2\| _{\ell ^2}\| f_3\| _{\ell ^\infty}.
\end{equation}
It can be shown by the H\"oder inequality as follows:
\begin{align*}
&\| \sum _{k_1+k_2+k_3=k}\frac{f_1(k_1)f_2(k_2)f_3(k_3)}{\langle k_1+k_2\rangle}\| _{\ell ^\infty}\\
&\le \| f_3\| _{\ell ^\infty}\Bigl ( \sum _{k_1,k_2\in \mathbf{Z}}\frac{|f_1(k_1)|^2}{\langle k_1+k_2\rangle ^{1-}\langle k_2\rangle ^{0+}}\Bigr )^{\frac{1}{2}} \Bigl ( \sum _{k_1,k_2\in \mathbf{Z}}\frac{\langle k_2\rangle ^{0+}|f_2(k_2)|^2}{\langle k_1+k_2\rangle ^{1+}}\Bigr )^{\frac{1}{2}}\\
&\lesssim \| f_1\| _{\ell ^2}\| \langle k\rangle ^{0+}f_2\| _{\ell ^2}\| f_3\| _{\ell ^\infty}.
\end{align*}

For the proof of \eqref{est:H'2}, we may focus on estimating
\[ \sup _{k\in \mathbf{Z}}\; \langle k\rangle ^{s+1}\!\!\sum _{\begin{smallmatrix} D_2(k)\setminus D_{2,3}(k)\\ |k_1|\ge |k_2|\end{smallmatrix}}\frac{\langle k_1 \rangle}{|\Phi |}|\hat{w}(k_1)||\hat{w}(-k_2)||\partial_t \hat{w}(k_3)| \] 
without loss of generality.
If $|\Phi |\gtrsim \langle k_1\rangle ^2\langle k_1+k_2 \rangle$, the desired estimate follows from \eqref{est:H'a} and \eqref{est:w_t}.
When $|\Phi |\gtrsim \langle k_1\rangle ^2\langle k_2+k_3 \rangle$, it suffices to apply the following estimates:
\begin{align*}
\| \sum _{k_1+k_2+k_3=k}\frac{f_1(k_1)f_2(k_2)f_3(k_3)}{\langle k_2+k_3\rangle}\| _{\ell ^\infty}\le &\; \| f_3\| _{\ell ^\infty}\sup _{k\in \mathbf{Z}}\sum _{k_1,k_2\in \mathbf{Z}}\frac{|f_1(k_1)|}{\langle k-k_1\rangle}|f_2(k_2)|\\
\lesssim &\; \| f_1\| _{\ell ^2}\| \langle k\rangle ^{\frac{1}{2}+}f_2\| _{\ell ^2}\| f_3\| _{\ell ^\infty}.
\end{align*}
Finally, in the case where $|\Phi |\gtrsim \langle k_1\rangle ^2\langle k_1+k_3 \rangle$, we notice that
\[ \frac{\langle k\rangle ^{s+1}\langle k_1 \rangle}{|\Phi |}\lesssim \frac{\langle k\rangle ^{\frac{1}{2}+}\langle k_1 \rangle ^{s-\frac{1}{2}-}}{\langle k_1+k_3\rangle}
\le \frac{\langle k_1\rangle ^{s-\frac{1}{2}-}}{\langle k_1+k_3\rangle}\Bigl ( \langle k_2\rangle ^{\frac{1}{2}+}+\langle k_1+k_3\rangle ^{\frac{1}{2}+}\Bigr ) . \]
Hence, the following estimates imply the claim:
\begin{align*}
&\| \sum _{k_1+k_2+k_3=k}\frac{\langle k_1\rangle ^{s-\frac{1}{2}-}f_1(k_1)\langle k_2\rangle ^{\frac{1}{2}+}f_2(k_2)f_3(k_3)}{\langle k_1+k_3\rangle}\| _{\ell ^\infty}\\
&\quad \le \| f_3\| _{\ell ^\infty}\sup _{k\in \mathbf{Z}}\sum _{k_1,k_2\in \mathbf{Z}}\langle k_1\rangle ^{s-\frac{1}{2}-}|f_1(k_1)|\frac{\langle k_2\rangle ^{\frac{1}{2}+}|f_2(k_2)|}{\langle k-k_2\rangle}\\
&\quad \lesssim \| \langle k\rangle ^{s}f_1\| _{\ell ^2}\| \langle k\rangle ^{\frac{1}{2}+}f_2\| _{\ell ^2}\| f_3\| _{\ell ^\infty},\\
&\| \sum _{k_1+k_2+k_3=k}\frac{\langle k_1\rangle ^{s-\frac{1}{2}-}f_1(k_1)f_2(k_2)f_3(k_3)}{\langle k_1+k_3\rangle ^{\frac{1}{2}-}}\| _{\ell ^\infty}\\
&\quad \le \| f_3\| _{\ell ^\infty}\sup _{k\in \mathbf{Z}}\sum _{k_1,k_2\in \mathbf{Z}}\langle k_1\rangle ^{s-\frac{1}{2}-}|f_1(k_1)|\frac{|f_2(k_2)|}{\langle k-k_2\rangle ^{\frac{1}{2}-}}\\
&\quad \lesssim \| \langle k\rangle ^{s}f_1\| _{\ell ^2}\| \langle k\rangle ^{0+}f_2\| _{\ell ^2}\| f_3\| _{\ell ^\infty}.
\end{align*}

This completes the proof.
\end{proof}

By means of Lemma~\ref{lem:multilinear'}, interpolation, and Lemma~\ref{lem:L2law}, we deduce from \eqref{est:hat-w'} that
\begin{equation}\label{est:hat-w'2}
\begin{split}
|\hat{u}_0(k)|\le C\Big[ &\;\frac{e^{-\frac{\Gamma}{4\pi}\| u_0\| _{L^2}^2kT}}{\langle k\rangle ^{s_1}}E_{s_1}^+(T)+\frac{\| u_0\| _{L^2}^{2-\frac{1}{2{s_1}}-}}{\langle k \rangle ^{{s_1}+1}}E_{s_1}^+(T)^{1+\frac{1}{2{s_1}}+}\\
&+\frac{\| u_0\| _{L^2}^{2-\frac{2}{{s_1}}-}}{\langle k \rangle ^{{s_1}+2}}E_{s_1}^+(T)^{1+\frac{2}{{s_1}}+}+\frac{1}{\| u_0\|_{L^2}^2\langle k \rangle ^{3{s_1}-\frac{1}{2}}}E_{s_1}^+(T)^3 \Big] 
\end{split}
\end{equation}
for any $k$ satisfying \eqref{cond:k}, where the constant $C>0$ depends only on the parameters in the equation and ${s_1}$.

\begin{remark}\label{rem:infl}
We observe that the first term on the right-hand side of \eqref{est:hat-w'2} has an arbitrary decay in $k$ as long as, say, $\| u_0\| _{L^2}^2kT\gtrsim k^{\frac{1}{4}}$, while the last two terms decay faster than $k^{-s_1-1}$.
To deduce norm inflation for $s=s_1+1$, we need to make the second term (which comes from the ``high$\times$low$\to$high'' type nonlinear interactions $\tilde{H}_1$) also decay faster than $k^{-s_1-1}$.
This forces us to choose initial data $u_0=u_{0,k}$ so that $\| u_{0,k}\|_{L^2}$ decays as $k\to \infty$.
That is why it is difficult to treat the case of non-zero $u^*(0)$ in the same way as Theorem~\ref{infl+}.
Even if we consider the equation for the difference $u-u^*$, there would remain a term with $\sup _t\| u(t)-u^*(t)\| _{H^{s_1}}$ (and without $\| u(0)-u^*(0)\| _{L^2}$), which would have no extra decay in $k$.
It might be possible to overcome this difficulty by applying an integration by parts once more to $\tilde{H}_1$.
\end{remark}

Now, we give a proof of Theorem~\ref{infl}.

\begin{proof}[Proof of Theorem~\ref{infl}]
Let $s,s_1$ satisfy \eqref{cond:infl}. 
We define the analytic function $\psi _{\varepsilon ,\tau}$ by
\begin{equation}\label{def:u_0n}
\psi _{\varepsilon ,\tau}(x):=\langle k_0\rangle ^{-\sigma (s)}+\frac{\varepsilon}{\sqrt{4\pi}} \langle k_0\rangle ^{-s}e^{ik_0x},\quad \sigma (s):=
\begin{cases}
\frac{2}{3}s-\frac{1}{2} &\text{if $1\le s\le \frac{5}{4}$,}\\
\frac{1}{3} &\text{if $s>\frac{5}{4}$,}
\end{cases}
\end{equation}
where $k_0$ is a large positive frequency to be chosen later.
Note that $\| \psi _{\varepsilon ,\tau}\| _{L^2}\sim \langle k_0\rangle ^{-\sigma (s)}$, and that $\| \psi _{\varepsilon ,\tau}\| _{H^s}\le \varepsilon$ if $k_0$ is sufficiently large.

Assume that there exists a solution $u\in C([0,\tau ];H^{s_1}(\mathbf{T}))$ to \eqref{3NLSR} with initial condition $u(0)=\psi _{\varepsilon ,\tau}$.
We need to show $E_{s_1}^+(\tau )\ge \varepsilon ^{-1}$ if $k_0$ is chosen appropriately.

Suppose that $E_{s_1}^+(\tau )<\varepsilon ^{-1}$.
We see that the condition \eqref{cond:k} (with $T$ replaced by $\tau$) is fulfilled if $k_0^{\frac{1}{3}}\ge C\varepsilon ^{-2}$. 
Then, from \eqref{est:hat-w'2}, at least one of the following four conditions holds:
\begin{align}
\varepsilon \langle k_0\rangle ^{-s}&\le Ce^{-C'\langle k_0\rangle ^{1-2\sigma (s)}\tau}\langle k_0\rangle ^{-s_1}E_{s_1}^+(\tau ),\label{cond:a1}\\
\varepsilon \langle k_0\rangle ^{-s}&\le C\langle k_0 \rangle ^{-\big[ s_1+1+\sigma (s)\big( 2-\frac{1}{2s_1}\big) \big] +}E_{s_1}^+(\tau )^{1+\frac{1}{2{s_1}}+},\label{cond:a2}\\
\varepsilon \langle k_0\rangle ^{-s}&\le C\langle k_0 \rangle ^{-\big[ s_1+2+\sigma (s)\big( 2-\frac{2}{s_1}\big) \big] +}E_{s_1}^+(\tau )^{1+\frac{2}{{s_1}}+},\label{cond:a3}\\
\varepsilon \langle k_0\rangle ^{-s}&\le C\langle k_0 \rangle ^{-\big[ 3s_1-\frac{1}{2}-2\sigma (s)\big]}E_{s_1}^+(\tau )^3.\label{cond:a4} 
\end{align}

If \eqref{cond:a1} holds, we have
\[ \varepsilon \langle k_0\rangle ^{-s}\le C\langle k_0\rangle ^{-5(1-2\sigma (s))-s_1}\tau ^{-5}E_{s_1}^+(\tau ).\]
Since $s<5(1-2\sigma (s))+s_1$ under the condition \eqref{cond:infl}, we can take $k_0$ sufficiently large so that
\[ \varepsilon \langle k_0\rangle ^{-s}\ge C\langle k_0\rangle ^{-5(1-2\sigma (s))-s_1}\tau ^{-5}\varepsilon^{-1}.\]
For such $k_0$ it must hold that $E_{s_1}^+(\tau )\ge \varepsilon^{-1}$.
The same argument can be applied to the other cases \eqref{cond:a2}--\eqref{cond:a4}; it suffices to check that
\[ s<\min \Big\{ s_1+1+\sigma (s)\big( 2-\frac{1}{2s_1}\big) ,\,s_1+2+\sigma (s)\big( 2-\frac{2}{s_1}\big) ,\,3s_1-\frac{1}{2}-2\sigma (s)\Big\} \]
under the condition \eqref{cond:infl}, which is easy to show.
Hence, we have $E_{s_1}^+(\tau )\ge \varepsilon^{-1}$ for any sufficiently large $k_0$, which contradicts our hypothesis.

This concludes the proof.
\end{proof}

%
%

\section{Existence of analytic solutions}

In this section, we show the unique local solvability of the Cauchy problem \eqref{3NLSR}-\eqref{ic} in the analytic function space.
We begin with the definition of the function space with which we work.

\begin{definition}\label{def:analytic}
For $r>0$, we define a Banach space $\mathcal{A}(r)$ by
\begin{equation*}
\mathcal{A}(r):=\bigl \{ f\in L^2 (\mathbf{T})\,\big| \, \| f\| _{\mathcal{A}(r)}:=\| e^{r|k|}\hat{f}(k)\| _{\ell ^1(\mathbf{Z})}<\infty \bigr \} .
\end{equation*}
\end{definition}

\begin{remark}
The function space $\mathcal{A}(r)$ was introduced by Ukai \cite [norm (2.6) and Definition 2.2 on page 143] {Uk} for the Boltzmann equation, by Kato and Masuda \cite [the definition of $A(r)$ on page 459] {KM} for a class of nonlinear evolution equations and by Foias and Temam \cite [(1.10) on page 361] {FT} for the incompressible Navier-Stokes equations.
Functions in $\mathcal{A}(r)$ are real analytic and have analytic extensions on the strip $\{ z\in \mathbf{C}|\,|\Im z|<r\}$ (see, e.g., \cite [Exercise 4.4 in Exercises for Section 4 on page 28] {Katz}).
In fact, for any $f\in \mathcal{A}(r)$ and positive integer $n$, we see that
\begin{align*}
\| \partial_x^n f\| _{L^\infty (\mathbf{T})}
&\lesssim \| |k|^n\hat{f}(k) \| _{\ell ^1(\mathbf{Z})}\le \| f\| _{\mathcal{A}(r)}\sup _{k\in \mathbf{Z}}|k|^ne^{-r|k|}\\
&=\| f\| _{\mathcal{A}(r)}\Bigl ( \frac{n}{r}\Bigr ) ^{n}\sup _{k\in \mathbf{Z}}\Bigl ( \frac{r}{n}|k|e^{-\frac{r}{n}|k|}\Bigr ) ^{n}\le  \| f\| _{\mathcal{A}(r)}\Bigl ( \frac{1}{r}\Bigr ) ^nn!,
\end{align*}
where at the last inequality we have used $n^n<n!e^{n}$ and $\sup _{\xi \ge 0}\xi e^{-\xi}=e^{-1}$.
\end{remark}

\begin{proposition}\label{prop:analytic}
Let $\alpha_j$, $j= 1, 2$ be two real numbers and let $r>0$.
For any $u_0\in \mathcal{A}(r)$, there exist $T>0$ such that the Cauchy problem \eqref{3NLSR}--\eqref{ic} has a unique solution $u\in C([-T,T];\mathcal{A}(r/2))$ on $(-T, T)$.
Moreover, $T$ can be chosen as
\begin{equation*}
T\gtrsim \min \{ 1,r\} \| u_0\| _{\mathcal{A}(r)}^{-2},
\end{equation*}
where the implicit constant does not depend on $r$ and $u_0$.
\end{proposition}

\begin{remark}
We do not have to assume \eqref{h1} in Proposition \ref{prop:analytic}.
Even when $\alpha_1= \alpha_2 = 0$, Proposition \ref{prop:analytic} holds.
\end{remark}

\begin{proof}
We will construct a solution by a fixed point argument on the associated integral equation
\begin{align}
u(t)=&\;U(t)u_0+\int _0^t U(t-t')\bigl [i\gamma_1 |u|^2u+\gamma_2 \partial_x (|u|^2u)-i\Gamma u\partial_x (|u|^2)\bigr ] (t')\,dt'\label{IntEq}\\
=:&\; \Psi [u_0](u)(t),\qquad t\in [-T ,T],\notag
\end{align}
where $U(t):=e^{t(\alpha_1 \partial_x^3 +i\alpha_2 \partial_x^2)}$.
We shall show that for $u_0\in \mathcal{A}(r)$, $\Psi [u_0]$ is a contraction on 
\begin{align}
B_{r,T}:=\Big\{ u\in C([-T,T]; \mathcal{A}(r/2))\, \Big| \  &u(t) \in \mathcal{A}\bigl (1- \frac {|t|} {2T} \bigr ), \ t \in [-T, T],  \notag \\
&|\!|\!| u |\!|\!| _{r,T}\le 2\| u_0\| _{\mathcal{A}(r)} \Bigr \} ,  \label{afsp} \\
|\!|\!| u |\!|\!| _{r,T}:=\big\| \sup _{|t|\le T}e^{r(1-\frac{|t|}{2T})|k|}|\hat{u}(t,k)|&\big\| _{\ell ^1(\mathbf{Z})} \notag
\end{align}
for suitable $T>0$.
Note that $\sup _{|t|\le T}\| u(t)\| _{\mathcal{A}(r/2)}\le |\!|\!| u|\!|\!| _{r,T}$.

Clearly, we have
\[ |\!|\!| U(t)u_0 |\!|\!| _{r,T}=\| u_0\| _{\mathcal{A}(r)}.\]
Next, we notice that 
\begin{align*}
&|\!|\!| \int _0^tU(t-t')[u_1\bar{u}_2u_3](t')\,dt'|\!|\!| _{r,T}\\
&\lesssim \big\| \sum _{k_1+k_2+k_3=k}\sup _{|t|\le T}e^{r(1-\frac{|t|}{2T})|k|}\int _0^t|\hat{u}_1(t',k_1)\hat{\bar{u}}_2(t',k_2)\hat{u}(t',k_3)|\,dt'\big\| _{\ell ^1}\\
&\le T \big\| \sum _{k_1+k_2+k_3=k}\sup _{|t'|\le T}e^{r(1-\frac{|t'|}{2T})(|k_1|+|k_2|+|k_3|)}|\hat{u}_1(t',k_1)\hat{\bar{u}}_2(t',k_2)\hat{u}(t',k_3)|\big\| _{\ell ^1}\\
&\lesssim T\prod _{j=1}^3|\!|\!| u_j|\!|\!| _{r,T}.
\end{align*}

For nonlinear terms with derivative, we observe that
\begin{align*}
e^{r(1-\frac{|t|}{2T})|k|}|k|\le &\; |k|e^{-\frac{r}{2T}(|t|-|t'|)|k|}\prod _{j=1}^3e^{r(1-\frac{|t'|}{2T})|k_j|},\\
e^{r(1-\frac{|t|}{2T})|k|}|k_1+k_2|\le &\; e^{r(1-\frac{|t|}{2T})|k_1+k_2|}|k_1+k_2|\cdot e^{r(1-\frac{|t'|}{2T})|k_3|}\\
\le &\; |k_1+k_2|e^{-\frac{r}{2T}(|t|-|t'|)|k_1+k_2|}\prod _{j=1}^3e^{r(1-\frac{|t'|}{2T})|k_j|}
\end{align*}
for $k=k_1+k_2+k_3$ and $0\le |t'|\le |t|\le T$.
Since a simple computation yields
\[
   \Bigl | \int_0^t |k| e^{-\frac r {2T} (|t| - |t'|) |k|} \ dt' \Bigr | \lesssim \frac T r,
\]
we thus obtain that
\begin{align*}
&|\!|\!| \int _0^tU(t-t')\bigl [\partial_x (u_1\bar{u}_2u_3)\bigr ] (t')\,dt'|\!|\!| _{r,T}\\
&\lesssim \big\| \sum _{k_1+k_2+k_3=k}\sup _{|t|\le T}\int _0^t |k|e^{-\frac{r}{2T}(|t|-|t'|)|k|}e^{r(1-\frac{|t'|}{2T})\sum _j|k_j|}\\[-10pt]
&\hspace{180pt} \times |\hat{u}_1(t',k_1)\hat{\bar{u}}_2(t',k_2)\hat{u}(t',k_3)|\,dt' \big\| _{\ell ^1}\\
&\lesssim \frac{T}{r}\big\| \sum _{k_1+k_2+k_3=k}\sup _{|t'|\le T}e^{r(1-\frac{|t'|}{2T})(|k_1|+|k_2|+|k_3|)}|\hat{u}_1(t',k_1)\hat{\bar{u}}_2(t',k_2)\hat{u}(t',k_3)|\big\| _{\ell ^1}\\
&\lesssim \frac{T}{r}\prod _{j=1}^3|\!|\!| u_j|\!|\!| _{r,T},
\end{align*}
and similarly, 
\[ |\!|\!| \int _0^tU(t-t')\bigl [u_3\partial_x (u_1\bar{u}_2)\bigr ] (t')\,dt'|\!|\!| _{r,T}\lesssim \frac{T}{r}\prod _{j=1}^3|\!|\!| u_j|\!|\!| _{r,T}.\]
Therefore, we have
\begin{align*}
|\!|\!| \Psi [u_0](u)|\!|\!| _{r,T}&\le \| u_0\| _{\mathcal{A}(r)}+CT(1+r^{-1})|\!|\!| u|\!|\!| _{r,T}^3,\\
|\!|\!| \Psi [u_0](u)-\Psi [u_0](v)|\!|\!| _{r,T}&\le CT(1+r^{-1})\bigl ( |\!|\!| u|\!|\!| _{r,T}^2+|\!|\!| v|\!|\!| _{r,T}^2\bigr ) |\!|\!| u-v|\!|\!| _{r,T}.
\end{align*}
Furthermore, it is easy to show that $\Psi [u_0](u)\in C([-T,T];\mathcal{A}(r/2))$ for $u_0\in \mathcal{A}(r)$ and $u\in B_{r,T}$.
Hence, $\Psi [u_0]$ is a contraction on $B_{r,T}$ if
\[ T\le c\min \{ 1,r\} \| u_0\| _{\mathcal{A}(r)}^{-2}\]
for some small constant $c>0$.
By Banach's fixed point theorem, we obtain a solution $u\in C([-T,T];\mathcal{A}(r/2))$ to the integral equation \eqref{IntEq}, which clearly solves \eqref{3NLSR}--\eqref{ic}.
Now the proof for the uniqueness of solutions is standard and so we omit it.
\end{proof}

\begin{remark}
(a) Even when the initial datum is a Gaussian pulse, it is open whether the solution given by Proposition \ref{prop:analytic} exists globally in time or not.
\par (b) The proof of Proposition \ref{prop:analytic} is based on the contraction mapping principle, which yields the continuous dependence of solutions on initial data in a sense.
This implies that the solution map is Lipschitz continuous from $u_0 \in \mathcal{A}(r)$ to $u \in B_{r,T}$, where $B_{r,T}$ is defined as in \eqref{afsp}.
\end{remark}

\begin{example}
Consider as initial data the rescaled periodic Gaussian $g_\lambda$ ($\lambda >0$) defined by
\[ \hat{g}_\lambda (k):=\lambda e^{-\lambda ^2k^2},\qquad k\in \mathbf{Z}.\]
We choose $r=\lambda$ and estimate the $\mathcal{A}(\lambda )$-norm of $g_\lambda$ as
\[ \| g_\lambda \| _{\mathcal{A}(\lambda )}\lesssim \int _0^\infty \lambda e^{-\lambda ^2\xi ^2+\lambda \xi}\,d\xi +\sup _{\xi \ge 0}\lambda e^{-\lambda ^2\xi ^2+\lambda \xi}\lesssim 1+\lambda .\]
Proposition~\ref{prop:analytic} then shows that if $0<\lambda \lesssim 1$, the corresponding solution $u_\lambda$ to \eqref{3NLSR} exists on $(-T_\lambda ,T_\lambda )$ with
\begin{equation}
    T_\lambda \gtrsim \lambda. \label{span}
\end{equation}
In most literature, numerical computations are carried out for five to ten times as long a period of time as the dispersion length (see, e.g., \cite[Figure 4.23 on page 112]{Agr}).
When  $\alpha_1 = 0$ and the initial datum is the rescaled periodic Gaussian pulse defined as above, the dispersion length $L_D$ is defined as $L_D = \lambda^2 / |\alpha_2|$ (see \cite[(4.4.2) in Section 4.4]{Agr}).
From \eqref{span}, it is presumed that the numerical solution for the ill-posed Cauchy problem \eqref{3NLSR}-\eqref{ic} may approximate the analytic solution given by Proposition \ref{prop:analytic} for as long a period of time as the length of  $\lambda = \frac {|\alpha_2|} {\lambda} L_D$.
We note that if $\lambda$ is small and $|\alpha_2| \sim 1$, this time range may be able to cover the period of time for which the numerical simulations are carried out in previous literature.
\end{example}

\noindent {\bf Concluding Remark}.
In \cite{Tsug}, Tsugawa introduced the notion of \lq\lq parabolic resonance'', by which some nonlinear terms could yield the smoothing type effect either forward or backward in time.
This might be applicable to nonlinear Schr\"odinger equations on the one dimensional torus, which leads to the ill-posedness.
But his proof is different from ours because our estimates are mainly done in the Fourier space while his proof proceeds in the $x$ variable space.

\section*{Acknowledgements}

The authors would like to thank the referee for suggesting that Hadamard's ill-posedness proof for the Cauchy problem of the Laplace equation should be applicable to the third order NLS with Raman scattering term.
In particular, the referee's suggestions enabled them to improve the ill-posedness results for solutions existing both forward and backward in the previous manuscript of this paper.
They are also grateful to Dr.~Tomoyuki Miyaji for fruitful discussions on the third order NLS with Raman scattering term and for showing them his interesting numerical simulations, which were helpful for their study of the ill-posedness.
The first author N.K is partially supported by JSPS KAKENHI Grant-in-Aid for Young Researchers (B) (16K17626).
The second author Y.T is partially supported by JSPS KAKENHI Grant-in-Aid for Scientific Research (B) (17H02853) and Grant-in-Aid for Exploratory Research (16K13770).




\end{document}